\documentclass[12pt,a4paper,reqno]{amsart}
\usepackage[T1]{fontenc}
\usepackage{lmodern}
\usepackage{amsmath,amssymb,amsthm,mathtools}
\usepackage[margin=28mm]{geometry}
\usepackage{%tikz-cd,
enumitem,microtype,needspace}

\usepackage{tikz}
\usetikzlibrary{positioning,intersections,cd,matrix, arrows, backgrounds}
\usetikzlibrary{shapes,shapes.geometric,shapes.misc}%% needed for rounded rectangle
\usetikzlibrary{decorations}
\usetikzlibrary{decorations.pathreplacing}%% instead of snakes
\def\Ar#1#2#3{\ar[from={#1}, to={#2}, #3]}

\usepackage[hidelinks]{hyperref}
\usepackage{calrsfs}

\setlist[enumerate]{leftmargin=*,itemsep=3pt}
\setlist[enumerate,1]{label=\textup{(\arabic*)},ref=\arabic*}
\setlist[enumerate,2]{label=\textup{(\alph*)},ref=\alph*}
\newcommand{\kk}{\Bbbk}
\renewcommand{\k}{\Bbbk}
\newcommand{\ZZ}{\mathbb Z}
\newcommand{\calC}{\mathcal C}
\newcommand{\calD}{\mathcal D}
\newcommand{\calR}{\mathcal R}
\newcommand{\calS}{\mathcal{S}}

\newcommand{\rad}{\operatorname{rad}}
\newcommand{\soc}{\operatorname{soc}}
\newcommand{\wt}{\operatorname{wt}}
\newcommand{\Ob}{\operatorname{Ob}}

\newcommand\ro{\rho}

\newcommand\ph{\phi}

\newcommand\Ga{\Gamma}

\newcommand\Ph{\Phi}

\newcommand{\ya}{\xrightarrow}

\newcommand{\Ker}{\operatorname{Ker}}
\newcommand{\Deck}{\operatorname{Deck}}
\newcommand{\sft}{\mathsf{t}}
\newcommand{\inv}{^{-1}}

\newcommand\id{1\kern-.25em{\text{{\rm l}}}}
\newcommand{\ang}[1]{\langle #1 \rangle}

\newtheorem{theorem}{Theorem}[section]
\newtheorem{proposition}[theorem]{Proposition}
\newtheorem{lemma}[theorem]{Lemma}
\newtheorem{corollary}[theorem]{Corollary}
\newtheorem*{maintheorem}{Theorem A}
\theoremstyle{definition}

\theoremstyle{remark}
\newtheorem{remark}[theorem]{Remark}
\numberwithin{equation}{section}

\allowdisplaybreaks[1]
\title[A nonstandard algebra of Euclidean type]{{A nonstandard selfinjective algebra of Euclidean type}}
\author[Hideto Asashiba]{{Hideto Asashiba}}
\address{Department of Mathematics, Faculty of Science,
Shizuoka University, 836 Ohya, Suruga-ku, Shizuoka, 422-8529, Japan}
\address{Institute for Advanced Study, KUIAS, Kyoto University,
Yoshida Ushinomiya-cho, Sakyo-ku, Kyoto 606-8501, Japan}
\address{Osaka Central Advanced Mathematical Institute,
3-3-138 Sugimoto, Sumiyoshi-ku, Osaka, 558-8585, Japan}
\email{asashiba.hideto@shizuoka.ac.jp}
\date{}
\hypersetup{pdfauthor={Hideto Asashiba},pdftitle={A nonstandard selfinjective algebra of Euclidean type},pdfsubject={A Euclidean orbit realization and an independent nonstandardness proof contradicting a published equivalence}}
\begin{document}
\begin{abstract}
\begingroup
Let $\kk$ be an algebraically closed field of characteristic different
from $2$. We realize the four-dimensional selfinjective algebra
$\Omega(1)$ as an admissible infinite cyclic orbit algebra of the
repetitive Kronecker category. Thus an algebra long stated to be
nonstandard is of Euclidean type in the orbit-algebra sense.
We also identify a gap in an earlier proof of its nonstandardness
and give an independent proof by determining all connected Galois
covering sources of $\Omega(1)$. Together, these results disprove the implication
from Euclidean type to standardness in the equivalence explicitly
stated by Bocian--Holm--Skowro\'nski (2004) and
Bocian--Skowro\'nski (2005).
They also contradict the standardness conclusion of the Proposition
in \S2.1 of Skowro\'nski (1989); its main theorem in \S1.5 assumes
standardness and is not contradicted by this example.
We construct a Galois covering of the repetitive Kronecker category
with simply connected source. Its composite with the orbit
projection is not Galois.
Finally, we give a sufficient lifting condition for standardness
of an orbit algebra. For an admissible infinite cyclic action on
the repetitive Kronecker category, an explicit two-by-two
diagonalizability criterion determines whether a generator
preserves a compatible pair system. Such a system supplies
the required lift through a grid covering.
\par\endgroup
\end{abstract}
\maketitle

\section{Introduction}\label{sec:introduction}

Throughout this article, $\kk$ is an algebraically closed field with
$\operatorname{char}\kk\ne2$, and algebras are finite-dimensional
over $\kk$.
We compose morphisms and paths from right to left: for morphisms
$f\colon x\to y$ and $g\colon y\to z$, the product $gf$ means $g\circ f$.
Bocian--Holm--Skowro\'nski \cite[p.~45]{BHS04} and
Bocian--Skowro\'nski \cite[p.~158]{BS05} explicitly state the
following equivalence for selfinjective algebras, both with reference
to Skowro\'nski \cite{Sko}:
\begin{quote}
A selfinjective algebra is of Euclidean type if and only if it is
standard and representation-infinite domestic.
\end{quote}
The same equivalence was repeated in \cite[Remark~1.3]{AIS} as a
consequence of the main theorem of \cite{Sko}.
One purpose of the present paper is to give a counterexample to its
``only if'' implication. More precisely, we construct a selfinjective
algebra of Euclidean type which is not standard.

We use standardness in the sense of \cite[\S1.2]{Sko}: an algebra is
{\emph{standard}} if it admits a Galois covering by a simply connected locally
bounded category. The covering and simple connectedness conventions
are given in Section~\ref{sec:conventions}.
{\emph{Representation-infinite}} means that there are infinitely many
isomorphism classes of indecomposable finite-dimensional modules.
We use domestic in the sense of \cite[\S1.1]{Sko}.\footnote{An algebra
$R$ is {\emph{domestic}} if there are an integer $N\ge1$ and
$R$--$\kk[X]$-bimodules $M_1,\ldots,M_N$, finitely generated and
free as right $\kk[X]$-modules, with the following property:
for all integers $d\ge1$, all but finitely many isomorphism classes
of indecomposable left $R$-modules of dimension $d$ have a
representative $M_i\otimes_{\kk[X]}V$ for some
$i\in\{1,\ldots,N\}$ and some finite-dimensional indecomposable
$\kk[X]$-module $V$. Here $X$ is a polynomial indeterminate.}

Euclidean type is understood in the orbit-algebra sense used in
\cite{BHS04,BS05}: a selfinjective algebra is of {\emph{Euclidean type}} if
it is isomorphic to $\widehat B/G$, where $B$ is a tilted algebra
of Euclidean type and $G$ is an admissible group of $\kk$-linear
automorphisms of its repetitive category $\widehat B$.
Here {\emph{tilted of Euclidean type}} means tilted from a hereditary algebra
whose quiver has an extended Dynkin underlying graph.
An {\emph{automorphism}} is an invertible $\kk$-linear functor and may permute
objects. {\emph{Admissibility}} means that the action on objects is free and
has finitely many object orbits. {\emph{Freeness}} means that, for all $g\in G$
and all objects $x$, the equality $gx=x$ implies $g=1$.
These conditions concern only the action on objects;
they do not require each automorphism to send a chosen arrow
to a scalar multiple of a chosen arrow.

We consider
\begin{equation}\label{eq:A}
 A=\kk\langle a,b\rangle/(a^2,\ ab+ba,\ b^2-ab).
\end{equation}
This is the four-dimensional algebra denoted by $\Omega(1)$ in
Bocian--Holm--Skowro\'nski \cite[\S1]{BHS}.
Its nonstandardness has long been stated in the literature; see
\cite[Theorem~2.1]{BHS} and
\cite[preprint version, \S4.36 and Proposition~4.37]{AKMW}.
Our first result is that this algebra is nevertheless of Euclidean
type: we exhibit it as an admissible infinite cyclic orbit algebra
of the repetitive Kronecker category. We then give an independent
proof of nonstandardness, because the argument in
\cite[Proposition~5.3]{BS} has a gap.

Let $K$ be the Kronecker algebra, with quiver
\begin{equation}\label{eq:Kronecker-quiver}
\begin{tikzcd}[column sep=huge]
 x\arrow[r,bend left=17,"\alpha"]\arrow[r,bend right=17,"\beta"']&y.
\end{tikzcd}
\end{equation}
It is hereditary of Euclidean type and hence tilted of Euclidean
type. Our main results are as follows.

\begin{maintheorem}\label{thm:main}
For the algebra $A$ in \eqref{eq:A}, the following statements hold.
\begin{enumerate}
\item\label{item:A-euclidean} $A$ is a basic, connected, four-dimensional selfinjective
algebra of Euclidean type. More precisely, there is an automorphism
$\delta$ of $\widehat K$ such that $\langle\delta\rangle$ is
admissible and infinite cyclic, and
\[
 A\cong\widehat K/\langle\delta\rangle.
\]
\item\label{item:A-sources} Set $\calC_\infty:=\widehat K$ and
$\calC_n:=\calC_\infty/\langle\delta^n\rangle$ for each integer $n\ge1$.
Every connected Galois covering of $A$ has source isomorphic to
$\calC_\infty$ or to $\calC_n$ for some integer $n\ge1$.
None of these categories is simply connected.
Hence $A$ is not standard.
\item\label{item:A-composite} There are connected Galois coverings $P$ and $Q$ in
\begin{equation}\label{eq:cover-chain}
\begin{tikzcd}[column sep=huge]
 \calD\arrow[r,"P"]\arrow[rr,bend right=25,"H=QP"']&
 \calC_\infty\arrow[r,"Q"]&A
\end{tikzcd}
\end{equation}
such that $\calD$ is simply connected and $H$ is a covering which
is not Galois.
\end{enumerate}
\end{maintheorem}

Parts (\ref{item:A-euclidean}) and (\ref{item:A-sources}) disprove the implication from Euclidean type to
standardness, and hence the ``only if'' direction of the equivalence
stated in \cite{BHS04,BS05,AIS}.
The reverse implication is not contradicted by this example.
In \cite{Sko}, the directly contradicted assertion is the
standardness conclusion of the Proposition in \S2.1, p.~180.
To identify the logical grounds of this equivalence, we combine
the main theorem in \S1.5 with that Proposition:
the main theorem gives the implication from standard
representation-infinite domestic type to Euclidean type,
whereas the Proposition asserts that the relevant Euclidean
orbit algebras are standard and representation-infinite domestic.
The problem exposed by our example is the standardness conclusion
of the Proposition, not the logical combination of these statements.
The Kronecker algebra also satisfies the additional complete-slice
condition in the definition of a Euclidean algebra in
\cite[\S1.3]{Sko}; we verify this in
Section~\ref{sec:repetitive}.
The main theorem assumes standardness.
Our nonstandard example is not a counterexample to
that theorem as stated. Nor does the failure of the displayed
equivalence by itself disprove an entire derived equivalence
classification.

We prove (\ref{item:A-euclidean}) first using the repetitive-quiver and orbit-category
calculations of \cite[\S1.3]{Asa99} and \cite[Theorem~8.1]{Asa}.
We present $\widehat K$ by its repetitive quiver and set
$\calC_\infty:=\widehat K$.
A change of arrow bases gives an index-shift automorphism $\delta$.
The orbit-category calculation then gives one vertex, two loops,
and exactly the defining relations of $A$.
This construction uses no nonstandardness assertion and no
classification of selfinjective algebras.

The proof of (\ref{item:A-sources}) requires more than citing the earlier
nonstandardness assertion. The proof of
\cite[Proposition~5.3, p.~4255]{BS} asserts that every bound-quiver
presentation of $\Omega(T)$ has trivial fundamental group.
For $\Omega(1)$, however, the presentation printed in
\cite[Proposition~5.1, p.~4253]{BS} has fundamental group $\ZZ$.
We calculate it in Section~\ref{sec:earlier}.
Thus this argument does not prove the nonstandardness of $\Omega(1)$.
We start with an arbitrary connected Galois covering of $A$, use
the induced group grading to determine its source, and prove that
none of the possible sources is simply connected. This gives an
independent proof of nonstandardness.

In (\ref{item:A-sources}), sources are classified up to isomorphism.
More precisely, for every connected Galois covering $F\colon\calC\to A$
there are an index $\ell\in\{1,2,3,\ldots,\infty\}$, an isomorphism
$\Ph\colon\calC_\ell\to\calC$ and an automorphism $\ph$ of $A$
such that the following diagram commutes:
\[
\begin{tikzcd}
 \calC_\ell\arrow[r,"\Ph","\cong"']\arrow[d,"Q_\ell"']&
 \calC\arrow[d,"F"]\\
 A\arrow[r,"\ph","\cong"']&A.
\end{tikzcd}
\]
The functors $Q_\ell$ are defined below; our Galois coverings are
strictly invariant in the sense of Section~\ref{sec:conventions}.
In general one cannot take $\ph=\id_A$; see
Proposition~\ref{prop:classification} and Remark~\ref{rem:target}.

Part (\ref{item:A-composite}) explains why a Galois covering of $\widehat K$ with
simply connected source need not yield a Galois covering of
$\widehat K/\langle\delta\rangle$ by composition with the orbit
projection. We also give a sufficient additional condition:
if a generator of an admissible infinite cyclic group of
automorphisms of a locally bounded category lifts through a
connected Galois covering of that category with simply connected
source, then the composite with the orbit projection is Galois
and the resulting orbit algebra is standard.
Theorem~\ref{thm:repair-lifts} proves a more general lifting
criterion. An {\emph{arrow line}} is the one-dimensional subspace spanned
by a chosen arrow. For the repetitive Kronecker category, we define
compatible pairs of such lines and show that a compatible pair
system invariant under the generator gives such a lift.
Proposition~\ref{prop:repair-monodromy} characterizes the existence
of an invariant compatible pair system by diagonalizability of an
explicit two-by-two matrix. In contrast, preservation of the family
of arrow lines in a chosen presentation does not ensure standardness
of the orbit algebra: $\delta$ permutes the original $a_i,b_i$ arrow
lines, but its orbit algebra is $A$.
Sections~\ref{sec:repetitive} and \ref{sec:compatibility} explain
this distinction.

Section~\ref{sec:conventions} gives the covering conventions and
the multiplication of $A$. Section~\ref{sec:repetitive} proves its
Euclidean orbit realization. We then examine the earlier
nonstandardness argument in Section~\ref{sec:earlier}, determine
all connected Galois sources in Section~\ref{sec:sources}, and
complete the proof of nonstandardness in
Section~\ref{sec:nonstandardness}.
Section~\ref{sec:composite} constructs the simply connected grid
covering and examines its composite with the orbit projection.
Section~\ref{sec:compatibility} proves the lifting criterion and
the sufficient compatibility condition.
Section~\ref{sec:conclusion} discusses the precise consequences
for the published equivalence and the standardness assertion.

\section{Covering conventions and the algebra}\label{sec:conventions}

We fix the covering and smash-product conventions and compute
the multiplication of $A$.

We write $\calC(x,y)$ for the space of morphisms
from $x$ to $y$.
The right-to-left convention also applies to multiplication in $A$.

All categories and functors in this article are $\kk$-linear.
A {\emph{line}} in a vector space means a one-dimensional linear subspace.
A {\emph{locally bounded}} category has pairwise nonisomorphic objects,
finite-dimensional local endomorphism algebras, and finite total
incoming and outgoing dimensions at each object. The latter condition
means that
\[
 \sum_{y\in\Ob\calC}\dim_{\kk}\calC(x,y)<\infty,
 \qquad
 \sum_{y\in\Ob\calC}\dim_{\kk}\calC(y,x)<\infty
\]
for all $x\in\Ob\calC$. A category is {\emph{connected}} if the unoriented
graph joining distinct objects with a nonzero morphism in either
direction is connected.
We identify a basic algebra $R$ with the category associated with
a fixed complete set $e_1,\ldots,e_t$ of primitive orthogonal
idempotents: its objects are the $e_i$, and its morphism space
from $e_i$ to $e_j$ is $e_jRe_i$ for all $i,j\in\{1,\ldots,t\}$.
In particular, $A$ has one object, denoted by $*$.
For a locally bounded category $\calR$ with finitely many objects,
its \emph{associated algebra} is
$\bigoplus_{x,y\in\Ob\calR}\calR(x,y)$, with multiplication given
by composition when the factors are composable and zero otherwise.
When $\calC/G$ has finitely many objects, its \emph{orbit algebra}
means the algebra associated with this orbit category.
For a one-object category this is just its endomorphism algebra.

A functor $F:\calC\to\calR$ is a {\emph{covering}} if it is surjective on
objects and, for all $x\in\Ob\calC$ and $r\in\Ob\calR$, both maps
\begin{align}\label{eq:dfn-covering1}
 \bigoplus_{\substack{y\in\Ob\calC\\Fy=r}}\calC(x,y)
 &\longrightarrow\calR(Fx,r),&
 (f_y)_y&\longmapsto\sum_y F(f_y),\\
\label{eq:dfn-covering2}
 \bigoplus_{\substack{y\in\Ob\calC\\Fy=r}}\calC(y,x)
 &\longrightarrow\calR(r,Fx),&
 (h_y)_y&\longmapsto\sum_y F(h_y)
\end{align}
are isomorphisms. We call these the {\emph{target-varying and source-varying covering
maps}}, respectively.
These terms specify the argument of the morphism space
that varies in the direct sum: the target in \eqref{eq:dfn-covering1}
and the source in \eqref{eq:dfn-covering2}.
The {\emph{object fiber}} over $r$ is
$\{y\in\Ob\calC \mid Fy=r\}$. A {\emph{connected covering}} has connected source.
We call a covering {\emph{proper}} if it is not an isomorphism.
A {\emph{deck transformation}} of $F$ is an automorphism $T$ of $\calC$ such
that $FT=F$; these form its {\emph{deck group}} under composition.
A {\emph{lift}} of an automorphism $g$ of $\calR$ through $F$ is an
automorphism $T$ of $\calC$ such that $FT=gF$.
\par\begingroup
A {\emph{strict action}} of a group $G$ on a category $\calC$ is a group
homomorphism $\rho\colon G\to\operatorname{Aut}(\calC)$, where
$\operatorname{Aut}(\calC)$ denotes the group of invertible functors
from $\calC$ to itself. We write $gx=\rho(g)(x)$ and
$g(f)=\rho(g)(f)$. Thus $\rho(1)=\id_{\calC}$ and
$\rho(gh)=\rho(g)\rho(h)$ for all $g,h\in G$, as equalities of
functors.
In compositions, we also write $g$ for $\rho(g)$.

For a category $\calC$ equipped with such an action, a functor
$F\colon\calC\to\calR$ is called {\emph{strictly $G$-invariant}} if
$F\rho(g)=F$ for all $g\in G$. Equivalently, each $\rho(g)$ is a
deck transformation of $F$. Throughout this article, invariance
means strict invariance in this sense.
\par\endgroup

We use the {\emph{classical orbit category}} of Gabriel.

Let a group $G$ act on a locally bounded category $\calC$ on the left,
and assume that the action is free on objects: for all $g\in G$ and
$x\in\Ob\calC$, the equality $gx=x$ implies $g=1$.
An action of $G$ on $\calC$ is called \emph{locally bounded}
if $\{g\in G\mid\calC(gx,y)\ne0\}$ is finite for all
$x,y\in\Ob\calC$.
The equivalent target-varying condition follows from
$\calC(gx,y)\cong\calC(x,g^{-1}y)$.
Here the action is locally bounded because $\calC$ is
locally bounded and $g\mapsto gx$ is injective for every
$x\in\Ob\calC$.

The objects of $\calC/G$ are the $G$-orbits in $\Ob\calC$.
For object orbits $X,Y$, a morphism $f:X\to Y$ is a row- and
column-finite matrix
\[
 f=(f_{v,u})_{v\in Y,\,u\in X},\qquad f_{v,u}\in\calC(u,v),
 \qquad f_{gv,gu}=g(f_{v,u}).
\]
{The last equality holds for all $g\in G$, $u\in X$ and $v\in Y$.
{\emph{Row- and column-finite}} means that each row and each column has only
finitely many nonzero entries.}
Composition is matrix multiplication.
This is the classical construction recalled in
\cite[Remark~2.2(1)]{Asa}.

Choose representatives $x\in X$ and $y\in Y$.
Taking the row indexed by $y$ gives the {\emph{source-varying identification}}
\begin{equation}\label{eq:orbit-convention}
 (\calC/G)(Gx,Gy)\cong\coprod_{g\in G}\calC(gx,y).
\end{equation}
Here $\coprod$ denotes an external direct sum of vector spaces.
Freeness makes every object in $Gx$ uniquely of the form $gx$.
Writing $f_g=f_{y,gx}$ for all $g\in G$, equivariance recovers
the other rows by $f_{hy,hgx}=h(f_g)$ for all $g,h\in G$.
For $f=(f_a)_{a\in G}\colon Gx\to Gy$ and
$h=(h_b)_{b\in G}\colon Gy\to Gz$, we have
$f_a\colon ax\to y$ and $h_b\colon by\to z$ for all $a,b\in G$.
For all $c\in G$, composition in these coordinates is
\begin{equation}\label{eq:orbit-composition}
 (hf)_c=\sum_{ba=c}h_b\,b(f_a).
\end{equation}
The summand in this formula is the composite
\begin{equation}\label{eq:source-composition-picture}
\begin{tikzcd}[column sep=huge]
 ba x\arrow[r,"b(f_a)"]&by\arrow[r,"h_b"]&z.
\end{tikzcd}
\end{equation}
These are the source-varying coordinates in
\cite[Definition~2.10(1) and Proposition~2.11]{Asa}.
Their agreement with the classical orbit category in the present
free-action setting is also recorded in \cite[Remark~8.2(2)]{Asa}.

The {\emph{canonical projection}} $\calC\to\calC/G$ sends $x$ to $Gx$.
It sends $f:x\to y$ to the equivariant matrix whose entry at
$(gy,gx)$ is $g(f)$ for each $g\in G$ and whose other entries are zero.
Under \eqref{eq:orbit-convention}, with the representatives $x$ and $y$ fixed,
the image of $f$ has component $f$ at $g=1$ and zero at every $g\ne 1$.
{The projection is strictly $G$-invariant.}

{Let $F\colon\calC\to\calR$ be a strictly
$G$-invariant functor. For objects $x$ and $y$ of $\calC$, define
the {\emph{source-varying covering map}} by}
\begin{equation}\label{eq:source-covering-map}
 F_{x,y}:\coprod_{g\in G}\calC(gx,y)\longrightarrow\calR(Fx,Fy),
 \qquad (f_g)_g\longmapsto\sum_g F(f_g).
\end{equation}
The corresponding {\emph{target-varying map}} is
\[
 \coprod_{g\in G}\calC(x,gy)\longrightarrow\calR(Fx,Fy),
 \qquad (h_g)_{g\in G}\longmapsto\sum_{g\in G}F(h_g).
\]
The $g$-summand of this target-varying domain is sent isomorphically
to the $g^{-1}$-summand of the source-varying domain by
$h_g\mapsto g^{-1}(h_g)\in\calC(g^{-1}x,y)$.
Strict invariance gives $F(g^{-1}(h_g))=F(h_g)$, so the two
covering maps agree under this direct-sum isomorphism.
Consequently the target-varying map is an isomorphism if and only if
$F_{x,y}$ is an isomorphism.
{For a more general formulation of these covering
maps, see \cite[Notation~1.5 and Proposition~1.6]{Asa}.}

A functor $F:\calC\to\calR$ is a \emph{Galois covering
with group $G$}, or a \emph{$G$-Galois covering} if it is a covering, is strictly
$G$-invariant, and $G$ acts transitively on each object
fiber of $F$.

Gabriel considers $G$ as a subgroup of $\operatorname{Aut}(\calC)$.
In our setting, if $\calC$ is nonempty, the free action
$\rho:G\to\operatorname{Aut}(\calC)$ is injective, and hence
$\rho$ identifies $G$ with the subgroup $\rho(G)$.
Moreover, coverings are surjective on objects in our convention.
Thus our definition agrees with Gabriel's definition applied
to $\rho(G)$ \cite[\S3.1, Remark]{Gabriel}.

Equivalently, there exists an isomorphism
$J:\calC/G\to\calR$ such that $F=J\pi$, where
$\pi:\calC\to\calC/G$ is the canonical covering
\cite[\S3.1, Remark]{Gabriel}.
In this case, the object fibers of $F$ are the $G$-orbits,
and the source-varying covering maps
\eqref{eq:source-covering-map} are isomorphisms.
This is the Galois covering used in the definition of
standardness in \cite[\S1.2]{Sko}.
All categories under consideration are locally bounded; in particular,
distinct objects are nonisomorphic.

For repetitive categories, a group action is admissible if it is free
on objects and has finitely many object orbits
\cite[\S1.4, p.~179]{Sko}. We call an automorphism {\emph{admissible}} if the
cyclic group it generates acts admissibly.

We recall the homotopy convention for a bound-quiver presentation
$\kk\mathsf Q/I$. A {\emph{walk}} in $\mathsf Q$ is a composable sequence
of arrows and formal inverse arrows; an {\emph{inverse arrow}} is traversed
in the reverse direction. A {\emph{relation}} is an element of $I$ that is
a linear combination of paths with the same source and target.
Its {\emph{support}} is the set of paths with nonzero coefficients.
A nonzero relation is {\emph{minimal}} if no nonempty proper subsum belongs
to $I$.
{\emph{Homotopy}} is the equivalence relation on walks generated by immediate
arrow--inverse cancellations and by identifying any two paths in the
support of a minimal relation, and preserved by concatenation on
either side. A {\emph{monomial zero relation}}, whose support has only one
path, does not identify two different paths.
The {\emph{bound-quiver fundamental group}} at a vertex consists of the
homotopy classes of closed walks at that vertex, with multiplication
induced by concatenation. A closed walk is {\emph{null-homotopic}} if it is
homotopic to the constant walk. 

A quiver is called \emph{acyclic} if it has no oriented cycle.
Following \cite[\S1.2]{Sko}, a connected locally bounded category is called
\emph{simply connected} if its quiver is acyclic and every bound-quiver
presentation has trivial fundamental group.

In particular, one presentation with nontrivial fundamental group is
enough to disprove simple connectedness. When we use the word
intrinsically, it emphasizes this requirement on every presentation;
it does not introduce a different notion of simple connectedness.

A \emph{$G$-grading} of a category $\mathcal B$ is a decomposition
$\mathcal B(x,y)=\bigoplus_{g\in G}\mathcal B_g(x,y)$ such that
$\id_x\in\mathcal B_1(x,x)$ and
$\mathcal B_g(y,z)\mathcal B_h(x,y)\subseteq\mathcal B_{gh}(x,z)$.
For a one-object algebra this is written $A=\bigoplus_{g\in G}A_g$.
An element of the degree-$g$ component is \emph{homogeneous of degree $g$};
a subspace is \emph{homogeneous} if it is the direct sum of its
intersections with the components.
The \emph{support} of a graded algebra is $\{g\in G\mid A_g\ne0\}$.
For $\ZZ$ we use additive degrees.

We use the \emph{left smash product} of \cite[Definition~5.2]{Asa}
and \cite[Definition~5.6.1]{AsaBook}. Its objects are
$x^{(a)}:=(x,a)$ for all $x\in\Ob\mathcal B$ and $a\in G$.
For all $x,y\in\Ob\mathcal B$ and $a,b\in G$, we identify
\begin{equation}\label{eq:smash-convention}
 (\mathcal B\#G)(x^{(a)},y^{(b)})=\mathcal B_{b^{-1}a}(x,y).
\end{equation}
Under this identification, a morphism is written simply as
$f:x^{(a)}\to y^{(b)}$, with $f\in\mathcal B_{b^{-1}a}(x,y)$,
and composition is that of $\mathcal B$.
For all $c\in G$, the left action sends $x^{(a)}$ to $x^{(ca)}$
and $f:x^{(a)}\to y^{(b)}$ to the morphism
$f:x^{(ca)}\to y^{(cb)}$ represented by the same element of $\mathcal B$.
The forgetful functor sends $x^{(a)}$ to $x$ and each morphism $f$
to the same element $f$ of $\mathcal B$.
For the locally bounded smash products used below, it is a Galois
covering \cite[Propositions~5.6 and 5.8]{Asa}.
For a homogeneous arrow $\alpha:x\to y$ of degree $g$, the
weighted-quiver calculation \cite[Theorem~6.2.18]{AsaBook} replaces
it by arrows $\alpha^{(a)}:x^{(ag)}\to y^{(a)}$ for all $a\in G$
and lifts the relations at each group index. Under
\eqref{eq:smash-convention}, the image of $\alpha^{(a)}$ is denoted
simply by $\alpha$.
This is the left-action version of \cite[Theorem~1.18]{AsaSmash}.
The notation $A\#G$ in that paper denotes the right-action version,
whose morphism degree is $ba^{-1}$; the book denotes that version
by $A\#^{\mathrm{op}}G$ \cite[\S6.3]{AsaBook}.
All smash products in this article use \eqref{eq:smash-convention}.

\Needspace{10\baselineskip}
\begin{lemma}\label{lem:algebra}
Set $s=ab$ and $J=\rad A$.
Then $1,a,b,s$ form a basis of $A$, and the multiplication is given
by the following table:
\begin{equation}\label{eq:table}
\begin{array}{c|cccc}
&1 &a & b & s\\
\hline
1 & 1 & a & b& s\\
a & a & 0 & s & 0\\
b & b & -s & s & 0\\
s & s & 0 & 0 & 0
\end{array}
\end{equation}
Moreover,
\[
 J=\kk a\oplus\kk b\oplus\kk s,\qquad
 J^2=\soc A=\kk s,\qquad J^3=0,
\]
and $A$ is basic, connected and selfinjective.
\end{lemma}
\begin{proof}
By the defining relations of $A$, we have $a^2=0,\, ab=s,\, ba=-s,\, b^2=s$.
It is easy to see that the remaining equalities follow from them.
Thus $1,a,b,s$ span $A$.
To show their linear independence, 
take a vector space $B$ with these four basis
symbols and define multiplication by the table \eqref{eq:table}, 
which shows that $1$ is the identity and that
every product of three elements of $\kk a+\kk b+\kk s$ is
zero, so this multiplication is associative.
Thus $B$ turns out to be an algebra.
Define an algebra homomorphism $f \colon \k\ang{a,b} \to B$ by $f(x):= x$
for each $x = a, b$.
Then it induces a surjective linear map $A \to B$ by the relations in \eqref{eq:table}, and hence
$\dim_\k A \ge \dim_\k B = 4$, which shows the required linear independence.
Thus $\{1,a,b,s\}$ is a basis of $A$.

The subspace $I:= \kk a+\kk b+\kk s$ is a nilpotent ideal of $A$ and the quotient $A/I$ is $\kk$.
Therefore $I = J$. The multiplication table gives $J^2=\kk s$ and
$J^3=0$. It also shows that the left and right annihilators of $J$ are
both $\kk s$, which proves the assertion about the socle.
In particular, $A$ is weakly symmetric, and hence a selfinjective algebra.

Finally, $A/J\cong\kk$ shows that
$A$ is basic and connected.
\end{proof}

Set $V=J/J^2$, and write $\bar u=u+J^2$ for all $u\in J$.
Since $J^3=0$, multiplication gives a well-defined bilinear form
\begin{equation}\label{eq:bilinear}
 B:V\times V\longrightarrow\kk,\qquad uv=B(\bar u,\bar v)s.
\end{equation}
In the basis $(\bar a,\bar b)$, the matrix of $B$ and its square
values, for all $x,y\in\kk$, are
\begin{equation}\label{eq:matrix}
 [B]=\begin{pmatrix}0&1\\-1&1\end{pmatrix},\qquad
 B(x\bar a+y\bar b,x\bar a+y\bar b)=y^2.
\end{equation}

A line $L\subseteq V$ is {\emph{isotropic}} for $B$ if $B(v,v)=0$ for all
$v\in L$. In particular, $B$ is nondegenerate and nonsymmetric, and
$\kk\bar a$ is its unique isotropic line. We will use these three
facts in Section~\ref{sec:sources}.

\section{The Euclidean orbit realization}\label{sec:repetitive}

We prove Theorem~A(\ref{item:A-euclidean}) by explicitly realizing $A$
as the orbit category $\widehat K/\langle\delta\rangle$.
We use the repetitive-quiver calculation in \cite[\S1.3]{Asa99}
and the orbit-category calculation in \cite[Theorem~8.1]{Asa}
(see also \cite[Theorem~6.1.2]{AsaBook}).

Let $K$ be the Kronecker algebra with arrows
$\alpha,\beta:x\to y$. It is representation-infinite and hereditary
of Euclidean type, hence tilted of Euclidean type.
Its indecomposable injective modules form a complete slice in the
preinjective component, as required in the definition of a Euclidean
algebra in \cite[\S1.3, p.~179]{Sko}.

The repetitive-quiver construction of \cite[\S1.3]{Asa99} gives
vertices $x_m,y_m$ and arrows
\begin{equation}\label{eq:repetitive-picture}
\begin{tikzcd}[column sep=large]
 x_m\arrow[r,shift left=3pt,"\alpha_m"]
 \arrow[r,shift right=3pt,"\beta_m"']&
 y_m\arrow[r,shift left=3pt,"r_m"]
 \arrow[r,shift right=3pt,"t_m"']&
 x_{m+1}\arrow[r,shift left=3pt,"\alpha_{m+1}"]
 \arrow[r,shift right=3pt,"\beta_{m+1}"']&y_{m+1}
\end{tikzcd}
\end{equation}
for all $m\in\ZZ$. The full subcategory on $x_m,y_m$ is the
\emph{$m$th layer}, identified with $K$ by
$(x_m,y_m,\alpha_m,\beta_m)\mapsto(x,y,\alpha,\beta)$.
The returning arrows $r_m,t_m$ correspond to the functionals dual
to $\alpha,\beta$ in $DK$, respectively, where
$D=\operatorname{Hom}_{\kk}(-,\kk)$.
Evaluation on the basis $e_x,e_y,\alpha,\beta$ of $K$ gives
\begin{align}
 r_m\alpha_m&=t_m\beta_m,&
 \alpha_{m+1}r_m&=\beta_{m+1}t_m,\label{eq:dual-pairing}\\
 t_m\alpha_m&=r_m\beta_m=0,&
 \beta_{m+1}r_m&=\alpha_{m+1}t_m=0.\nonumber
\end{align}
The common values in the first line are dual to $e_x$ and $e_y$,
respectively, and are nonzero. These are the defining quadratic
relations of $\widehat K$: all paths of length three already vanish
by these relations, and the surviving paths give the basis in
each layer and the dual basis between consecutive layers.

Set $\calC_\infty:=\widehat K$. For all $m\in\ZZ$, relabel its
vertices and arrows by
\begin{equation}\label{eq:repetitive-arrows}
 \begin{gathered}
 x_m=2m,\qquad y_m=2m+1,\\
 c_{2m}=\alpha_m,\quad d_{2m}=\beta_m,\quad
 c_{2m+1}=t_m,\quad d_{2m+1}=r_m.
 \end{gathered}
\end{equation}
For all $i\in\ZZ$, we have $c_i,d_i:i\to i+1$, and
\eqref{eq:dual-pairing} becomes
\begin{equation}\label{eq:cd-products}
 c_{i+1}c_i=d_{i+1}d_i=0,\qquad
 c_{i+1}d_i=d_{i+1}c_i\ne0.
\end{equation}
A family of ordered bases $(u_i,v_i)$ of
$\calC_\infty(i,i+1)$ is called a family of \emph{repetitive bases}
if $u_{i+1}u_i=v_{i+1}v_i=0$ and
$u_{i+1}v_i=v_{i+1}u_i\ne0$ for every $i\in\ZZ$.
In particular, $(c_i,d_i)_{i\in\ZZ}$ is such a family.
The convention for numbering repetitive layers in \cite[\S1.4]{Sko}
is obtained by replacing $m$ with $-m$.

To express the relations in the form defining $A$, put, for all $i\in\ZZ$,
\[
 a_i=(-1)^i c_i,\qquad b_i=d_i+i(-1)^i c_i,
\]
where integers are evaluated in $\kk$. Equivalently,
\begin{equation}\label{eq:basis}
 c_i=(-1)^i a_i,\qquad d_i=b_i-i a_i.
\end{equation}
With $w_i:=c_{i+1}d_i=d_{i+1}c_i$, direct substitution gives
\[
 a_{i+1}a_i=0,\qquad
 a_{i+1}b_i=(-1)^{i+1}w_i,\qquad
 b_{i+1}a_i=(-1)^i w_i,\qquad
 b_{i+1}b_i=(-1)^{i+1}w_i.
\]
Thus the same category has the quiver
\begin{equation}\label{eq:chain-picture}
\begin{tikzcd}[column sep=large]
 \cdots\arrow[r,shift left=3pt]\arrow[r,shift right=3pt]&i-1
 \arrow[r,shift left=3pt,"a_{i-1}"]
 \arrow[r,shift right=3pt,"b_{i-1}"']&i
 \arrow[r,shift left=3pt,"a_i"]
 \arrow[r,shift right=3pt,"b_i"']&i+1\arrow[r,shift left=3pt]\arrow[r,shift right=3pt]&\cdots
\end{tikzcd}
\end{equation}
with the following defining relations for all $i\in\ZZ$:
\begin{equation}\label{eq:chain-relations}
 a_{i+1}a_i=0,\qquad
 a_{i+1}b_i+b_{i+1}a_i=0,\qquad
 b_{i+1}b_i-a_{i+1}b_i=0.
\end{equation}
They are the indexed versions, in the same order, of
$a^2$, $ab+ba$ and $b^2-ab$ in \eqref{eq:A}.
For all $i\in\ZZ$, set $s_i:=a_{i+1}b_i=(-1)^{i+1}w_i$.
For all $i,j\in\ZZ$, the morphism spaces are
\begin{equation}\label{eq:chain-spaces}
 \calC_\infty(i,j)=
 \begin{cases}
 \kk\id_i,&j=i,\\
 \kk a_i\oplus\kk b_i,&j=i+1,\\
 \kk s_i,&j=i+2,\\
 0,&\text{otherwise}.
 \end{cases}
\end{equation}
In particular, $\calC_\infty$ is locally bounded.

Define $\delta$ by the following assignments for all $i\in\ZZ$:
\begin{equation}\label{eq:delta}
 \delta(i)=i+1,\qquad \delta(a_i)=a_{i+1},\qquad
 \delta(b_i)=b_{i+1},\qquad \delta(s_i)=s_{i+1}.
\end{equation}
Since \eqref{eq:chain-relations} is independent of $i$, these
assignments define an automorphism, with inverse given by shifting
indices by $-1$.
The infinite cyclic group $\langle\delta\rangle$ acts freely on
objects with one object orbit, so it is admissible.

Apply the orbit-category calculation of \cite[Theorem~8.1]{Asa},
and pass to the classical orbit category as in
\cite[Remark~8.2(2)]{Asa}.
There is one vertex orbit and two arrow orbits, represented by
$a_i$ and $b_i$. Eliminating the isomorphisms that identify vertices
in an orbit leaves one vertex $*$, two loops $a,b$, and the relations
\[
 a^2=0,\qquad ab+ba=0,\qquad b^2-ab=0.
\]
Thus the calculated classical orbit category is precisely $A$.
Under this identification its canonical projection is
$Q:\calC_\infty\to A$, given, for all $i\in\ZZ$, by
\begin{equation}\label{eq:Q}
 Q(i)=*,\qquad Q(a_i)=a,\qquad Q(b_i)=b,\qquad Q(s_i)=s.
\end{equation}
In particular, $Q$ is a Galois covering and $Q\delta=Q$.
This proves Theorem~A(\ref{item:A-euclidean}), which we record as follows.

\begin{corollary}\label{cor:euclidean}
The algebra $A$ is of Euclidean type. More precisely,
$A\cong\widehat K/\langle\delta\rangle$, where
$\langle\delta\rangle$ is admissible and infinite cyclic.
\end{corollary}

For later use, give $A$ the $\ZZ$-grading
$\deg a=\deg b=1$, $\deg s=2$.
The smash-product calculation
\cite[Theorem~6.2.18]{AsaBook} gives the same presentation
\eqref{eq:chain-picture}--\eqref{eq:chain-relations}:
with our left-action convention,
$\calC_\infty\cong A\#\ZZ$ by $i\mapsto *^{(-i)}$.
For all $i,j\in\ZZ$, the morphism space from $i$ to $j$ is $A_{j-i}$.
The functor $Q$ is the forgetful functor, and $\delta$ corresponds
to the action of $-1$.

For a chosen arrow $u:x\to y$ in a presentation, its chosen arrow
line is $\kk u\subseteq\calC(x,y)$, where $u$ also denotes its image
in the category. An automorphism may send it to a linear combination
of \emph{parallel arrows}, namely arrows with the same source and target.
In the repetitive bases, \eqref{eq:basis} gives, for all $i\in\ZZ$,
\begin{equation}\label{eq:mixing}
 \delta(c_i)=-c_{i+1},\qquad
 \delta(d_i)=d_{i+1}+(-1)^{i+1}c_{i+1}.
\end{equation}
Equivalently, for all $m\in\ZZ$,
\[
 \delta(\alpha_m)=-t_m,\quad
 \delta(\beta_m)=r_m-t_m,\quad
 \delta(r_m)=\beta_{m+1}+\alpha_{m+1},\quad
 \delta(t_m)=-\alpha_{m+1}.
\]
Also $\delta(w_i)=-w_{i+1}$ for all $i\in\ZZ$, so these formulas determine $\delta$
on every morphism space.
The mixed term in \eqref{eq:mixing} will be used in
Section~\ref{sec:composite}.
Although $\delta$ sends $\kk a_i$ to $\kk a_{i+1}$ and
$\kk b_i$ to $\kk b_{i+1}$, it does not preserve the corresponding
$c,d$ arrow lines.
Preservation of chosen arrow lines in one presentation is different
from preservation of a compatible pair system, as defined in
Section~\ref{sec:compatibility}.

\section{The gap in the earlier nonstandardness argument}\label{sec:earlier}

The Euclidean realization is now established. We next examine the
printed proof of the nonstandardness assertion for $A$ before
giving an independent proof in Sections~\ref{sec:sources} and
\ref{sec:nonstandardness}.
We use the bound-quiver homotopy convention fixed in
Section~\ref{sec:conventions}.

Let $\mathsf L$ be the quiver with one vertex and two loops $\gamma,\eta$.
The local case in \cite[Proposition~5.1, p.~4253]{BS}, with the parameter
set equal to $-1$ as on p.~4255, is
\begin{equation}\label{eq:BS-presentation}
 \Omega(1)=\kk\mathsf L/I,\qquad
 I=(\gamma^2-\gamma\eta,\ \eta^2,\ \gamma\eta+\eta\gamma).
\end{equation}
We have changed the name of the second loop from $\delta$ to $\eta$ to
reserve $\delta$ for our covering automorphism.
The substitutions $\gamma=b$ and $\eta=-a$ identify this algebra with $A$.
Indeed,
\[
 \gamma^2-\gamma\eta\longmapsto b^2+ba=0,\qquad
 \eta^2\longmapsto a^2=0,\qquad
 \gamma\eta+\eta\gamma\longmapsto-ba-ab=0.
\]
Conversely, $a=-\eta$ and $b=\gamma$ give the inverse substitution.
The four-dimensional local case is also included explicitly in
\cite[Corollary~6.2, p.~4257]{BS}.

\begin{proposition}\label{prop:BS-group}
For the presentation in \eqref{eq:BS-presentation},
\[
 \pi_1(\mathsf L,I)\cong\ZZ.
\]
\end{proposition}
\begin{proof}
The paths $\gamma^2$, $\gamma\eta$ and $\eta\gamma$ are nonzero,
as also follows from the multiplication table in Lemma~\ref{lem:algebra}.
Thus $\gamma^2-\gamma\eta$ and $\gamma\eta+\eta\gamma$
are minimal relations and give
\begin{equation}\label{eq:BS-homotopy}
 \gamma\gamma\sim\gamma\eta,
 \qquad \gamma\eta\sim\eta\gamma.
\end{equation}
The relation $\eta^2=0$ adds no homotopy relation.
We show that every minimal relation identifies only paths already
identified by \eqref{eq:BS-homotopy}.

The ideal $I$ is homogeneous in path length: if
$f=\sum_{n\geq 0}f_n\in I$, where $f_n$ is a linear combination
of paths of length $n$, then $f_n\in I$ for all $n\geq 0$.
All paths of length three
belong individually to $I$. For example,
\[
 \gamma^2\eta=\gamma\eta^2=0,\qquad
 \gamma\eta\gamma=-\gamma^2\eta=0,\qquad
 \gamma^3=\gamma\eta\gamma=0,
\]
and the remaining paths have a factor $\eta^2$ or reduce to these paths.
The same holds for every longer path.
The ideal $I$ has no nonzero components in path lengths $0$ or $1$,
and every path of length at least $3$ belongs to $I$ individually.
Hence every minimal relation with at least two terms is a linear
combination of paths of length two. The coefficient of $\eta^2$
must be zero: otherwise, that term alone would be a nonempty proper
subsum belonging to $I$, contrary to minimality.
Its paths therefore lie among
$\gamma^2,\gamma\eta,\eta\gamma$, which are identified by
\eqref{eq:BS-homotopy}.
Consequently
\begin{equation}\label{eq:BS-group}
 \pi_1(\mathsf L,I)
 =\langle\gamma,\eta\mid
       \gamma^2=\gamma\eta,\ \gamma\eta=\eta\gamma\rangle.
\end{equation}
Cancellation in the first relation gives $\gamma=\eta$.
After this substitution the second relation imposes no condition.
Thus \eqref{eq:BS-group} is infinite cyclic.
One can also check the infinitude directly: path length, with inverse
arrows counted negatively, induces a homomorphism to $\ZZ$ which sends
each loop to $1$.
\end{proof}

The proof of \cite[Proposition~5.3]{BS} first notes that the quiver has
oriented cycles. 
It then asserts that every bound-quiver presentation has trivial
fundamental group. If this assertion were true, every connected
Galois covering $\calC\to A$ would have trivial Galois group,
since that group is a quotient of the fundamental group of a
presentation of $A$ adapted to the covering. Hence $\calC\cong A$,
and the oriented cycles in the quiver of $A$ would exclude a
simply connected source. However,
Proposition~\ref{prop:BS-group} disproves the assertion about
fundamental groups for the local algebra.
The nonstandardness conclusion will be established independently
below by determining all possible Galois sources.

\section{All connected Galois covering sources}\label{sec:sources}

We now begin the independent proof of Theorem~A(\ref{item:A-sources}).
We determine all connected Galois sources using the multiplication
of $A$, without using the defective fundamental-group assertion
examined in Section~\ref{sec:earlier}.

Let $F:\calC\to A$ be a connected Galois covering with group $G$.
We do not assume that $G$ is abelian, finite or finitely generated.
Since $A$ has one object, the action of $G$ on $\Ob\calC$ is free and
transitive. Choose $x\in\Ob\calC$.
The induced-grading construction
\cite[Lemma~5.13 and Corollary~5.14]{Asa} gives
\begin{equation}\label{eq:grading}
 A=\bigoplus_{g\in G} A_g,
 \qquad A_g=F\bigl(\calC(gx,x)\bigr).
\end{equation}
With right-to-left composition the degree order is $gh$:
if $u:gx\to x$ and $v:hx\to x$, then
$u\,g(v):ghx\to x$ has image $F(u)F(v)$.
The smash-product reconstruction \cite[Theorem~5.10]{Asa},
in these coordinates, is the isomorphism
$\calC\cong A\#G$ given on objects by $gx\mapsto *^{(g)}$ and on
morphisms by the restrictions of $F$:
\begin{equation}\label{eq:recovered-spaces}
 \calC(gx,hx)\ \xrightarrow{\ F\ }\
 (A\#G)(*^{(g)},*^{(h)})=A_{h^{-1}g}.
\end{equation}
Thus its composite with the forgetful functor $A\#G\to A$ is $F$.
Here the general reconstruction is an isomorphism: the free
transitive action gives a bijection on objects, and
\eqref{eq:grading} gives a bijection on each morphism space.

\begin{lemma}\label{lem:homogeneous}
The ideals $J$ and $J^2$ are homogeneous for \eqref{eq:grading}.
\end{lemma}
\begin{proof}
We first show that $A_g\subseteq J$ for $g\ne1$.
Assume that $u\in A_g\setminus J$. Since $A$ is local, $u$ is a unit.
Write $u^{-1}=\sum_{h\in G}w_h$, where $w_h\in A_h$
for all $h\in G$, and put $v:=w_{g^{-1}}$.
Since $uw_h\in A_{gh}$ and $w_hu\in A_{hg}$,
the degree-$1$ components of $uu^{-1}=1=u^{-1}u$ give
$uv=1=vu$.
By \eqref{eq:recovered-spaces}, the restrictions
\[
 F:\calC(gx,x)\xrightarrow{\sim}A_g,
 \qquad
 F:\calC(x,gx)\xrightarrow{\sim}A_{g^{-1}}
\]
are isomorphisms. Thus there are unique morphisms
\[
 \widetilde u:gx\longrightarrow x,
 \qquad
 \widetilde v:x\longrightarrow gx
\]
such that $F(\widetilde u)=u$ and $F(\widetilde v)=v$.
Since
\[
 F(\widetilde u\widetilde v)=uv=1=F(\id_x),
 \qquad
 F(\widetilde v\widetilde u)=vu=1=F(\id_{gx}),
\]
and $F$ is injective on each morphism space, we obtain
\[
 \widetilde u\widetilde v=\id_x,
 \qquad
 \widetilde v\widetilde u=\id_{gx}.
\]
Hence $gx\cong x$. Distinct objects of $\calC$ are
nonisomorphic, so $gx=x$. The freeness of the $G$-action
then gives $g=1$, contrary to our assumption.

It follows that an element of $A$ belongs to $J$ if and only if its
identity-degree component belongs to $J$. Therefore
\[
 J=(J\cap A_1)\oplus\bigoplus_{g\ne1}A_g.
\]
This proves that $J$ is homogeneous. Products of homogeneous components
then show that $J^2$ is homogeneous.
\end{proof}
The induced grading on $V=J/J^2$ is
\[
 V=\bigoplus_{g\in G}V_g,\qquad
 V_g=((J\cap A_g)+J^2)/J^2.
\]
In particular, every homogeneous class of degree $g$ has a lift
in $J\cap A_g$.

\begin{lemma}\label{lem:one-degree}
The graded vector space $V=J/J^2$ is concentrated in one degree $t\in G$.
\end{lemma}
\begin{proof}
Since $J^2$ is one-dimensional and homogeneous, there is some $q\in G$
such that $J^2\subseteq A_q$.
Assume, for a contradiction, that $V$ is not concentrated in one degree.
Since $\dim V=2$, there are homogeneous basis vectors $u,v$ of distinct
degrees $g,h$.
Products of these classes are computed in $J^2$ using lifts in
$J$; they are independent of the lifts because $J^3=0$.
To determine their degrees, choose homogeneous lifts of degrees
$g$ and $h$, respectively. We derive a contradiction in two claims.

\smallskip\noindent
\textit{Claim 1. $u^2 = 0$ or $v^2 = 0$.}
Assume that both are nonzero. Then $g^2=h^2=q$.
At least one of $uv,vu$ is nonzero. Indeed, otherwise the matrix of $B$
in the basis $(u,v)$ would be diagonal and $B$ would be symmetric,
contrary to \eqref{eq:matrix}.
If $uv\ne0$, then $gh=q=g^2$, so $h=g$ by cancellation.
If $vu\ne0$, then $hg=q=g^2$, and cancellation again gives $h=g$.
Both cases contradict the choice of $u,v$.

\smallskip\noindent
\textit{Claim 2. $u^2 \ne 0$ and $v^2 \ne 0$.}
Assume that $u^2=0$; the other case is obtained by interchanging $u,v$.
The isotropic line of $B$ is unique, so $v^2\ne0$.
Moreover, $uv$ and $vu$ are both nonzero.
Indeed, the matrix of $B$ in this basis has upper-left entry zero,
and either off-diagonal entry being zero would make it singular.
Thus $gh=q=h^2$, and right cancellation gives $g=h$.
This is a contradiction.

Claims 1 and 2 exclude all possibilities for the two squares.
Therefore $V$ is concentrated in one degree.
\end{proof}

For each integer $n\ge1$, give $A$ the $\ZZ/n\ZZ$-grading
$\deg a=\deg b=\overline1$, obtained by reducing its $\ZZ$-grading
modulo $n$, and set
\[
 \calC_n:=A\#(\ZZ/n\ZZ).
\]
By the weighted-quiver calculation \cite[Theorem~6.2.18]{AsaBook},
the identification $i\leftrightarrow *^{(-i)}$ gives the presentation
\eqref{eq:chain-picture}--\eqref{eq:chain-relations} with indices
modulo $n$. Components of degrees that coincide modulo $n$ are
summed: for $n=2$ the socle morphism is an endomorphism, and for
$n=1$ the identity, the two arrows and the socle morphism all
belong to the same endomorphism algebra.
In particular, $\calC_1=A$.
The orbit-category calculation also gives
\[
 \calC_n\cong\calC_\infty/\langle\delta^n\rangle.
\]
We identify these categories by their indexed vertices and arrows.
Write $Q_n:\calC_n\to A$ for the canonical forgetful Galois covering,
whose values are \eqref{eq:Q}, and put $Q_\infty=Q$.

\begin{proposition}\label{prop:classification}
Let $F:\calC\to A$ be a connected Galois covering with group $G$.
Then $G$ is cyclic. For some
$\ell\in\{1,2,3,\ldots,\infty\}$ there are isomorphisms
$\Ph:\calC_\ell\to\calC$ and $\ph:A\to A$ for which
\begin{equation}\label{eq:classification-square}
\begin{tikzcd}[column sep=huge]
 \calC_\ell\arrow[r,"\Ph","\sim"']\arrow[d,"Q_\ell"']&
 \calC\arrow[d,"F"]\\
 A\arrow[r,"\ph","\sim"']&A
\end{tikzcd}
\end{equation}
commutes. The group $G$ is infinite cyclic when $\ell=\infty$, and has
order $\ell$ otherwise.
\end{proposition}
\begin{proof}
By Lemma~\ref{lem:one-degree}, $V=V_t$ for some $t\in G$.
Let $\pi:J\to V=J/J^2$ be the quotient map, and recall that
$\bar a=\pi(a)$ and $\bar b=\pi(b)$.
By the definition of the quotient grading,
\[
 V_t=((J\cap A_t)+J^2)/J^2
     =\pi(J\cap A_t).
\]
Hence the restriction $\pi:J\cap A_t\to V_t$ is surjective.
Choose $a',b'\in J\cap A_t$ such that
$\pi(a')=\bar a$ and $\pi(b')=\bar b$.
Since $J^2=\kk s$, these elements have the form
\begin{equation}\label{eq:lift-change}
 a'=a+\alpha s,\qquad b'=b+\beta s
 \qquad{\text{for some }\alpha,\beta\in\kk}.
\end{equation}
We show the assertion in three steps.

\smallskip\noindent
\textit{Claim 1. We have $G=\langle t\rangle$.}
Since $sJ=Js=0$, multiplication gives
\begin{equation}\label{eq:lift-products}
 (a')^2=0,\qquad a'b'=s,\qquad b'a'=-s,\qquad (b')^2=s.
\end{equation}
Thus $s$ is homogeneous of degree $t^2$.
Hence
$$
A = \k 1 + \k a' + \k b' + \k s \subseteq A_1 + A_t + A_{t^2} \subseteq \sum_{g \in \ang{t}}A_g \subseteq A, 
$$
so $A = \sum_{g \in \ang{t}}A_g$.
Thus $A_g = 0$ if $g \not\in \ang{t}$.
Therefore, by \eqref{eq:recovered-spaces},
\[
 \calC(gx,hx)\cong A_{h^{-1}g}=0
 \qquad\text{if }h^{-1}g\notin\langle t\rangle.
\]
Thus, if $gx$ and $hx$ are joined by a nonzero morphism
in either direction, then
$g \in h\langle t\rangle$.
Since $\calC$ is connected, every object $gx$ can be joined
to $x$ by a finite sequence of nonzero morphisms,
with either direction allowed at each step.
Consequently, $g \in \langle t\rangle$
for all $g\in G$, and hence $G=\langle t\rangle$.

\smallskip\noindent
\textit{Claim 2. The assignments $a\mapsto a'$ and $b\mapsto b'$
define an automorphism $\ph$ of $A$.}
The products in \eqref{eq:lift-products} are the defining relations of
$A$, and $a'b'=s$. Thus the assignments define a homomorphism with
$\ph(s)=s$. Its inverse is given by
$a\mapsto a-\alpha s$ and $b\mapsto b-\beta s$.
This proves the claim.

\smallskip\noindent
\textit{Claim 3. There is an isomorphism $\Ph$ for which
$F\Ph=\ph Q_\ell$.}
Identify $G=\langle t\rangle$ with $\ZZ$ if $t$ has infinite
order and with $\ZZ/n\ZZ$ if its order is $n$.
The automorphism $\ph$ is a graded isomorphism from the
grading $\deg a=\deg b=t$ to the grading \eqref{eq:grading}.
It therefore induces an isomorphism of their smash products.
Composing it with the inverse reconstruction
\eqref{eq:recovered-spaces} gives $\Ph:\calC_\ell\to\calC$,
explicitly
\begin{equation}\label{eq:Psi-objects}
 \Ph(i)=t^{-i}x,\qquad
 F\Ph(a_i)=a',\quad F\Ph(b_i)=b',\quad F\Ph(s_i)=s.
\end{equation}
These assignments are made for all $i\in\ZZ$ or all $i\in\ZZ/n\ZZ$,
respectively.
Naturality of the forgetful functors gives
$F\Ph=\ph Q_\ell$.
\end{proof}

\begin{remark}\label{rem:target}
In Proposition~\ref{prop:classification}, one cannot always
choose the isomorphism $\Phi$ so that $F\Phi=Q_\ell$.
To see this, let $\phi:A\to A$ be the automorphism defined by
$\phi(a)=a+s$ and $\phi(b)=b$, and fix the Galois covering
\[
 F:=\phi Q:\calC_\infty\longrightarrow A.
\]
Suppose that there is an automorphism $\Phi$ of $\calC_\infty$
such that $F\Phi=Q$.
For a locally bounded $\kk$-category $\calC$, put
\[
 \operatorname{Irr}_{\calC}(x,y)
 :=\rad\calC(x,y)/\rad^2\calC(x,y)
 \qquad\text{for all }x,y\in\Ob\calC.
\]
This is the \emph{irreducible morphism space} from $x$ to $y$;
$\rad^2\calC(x,y)$ consists of sums of composites of two
radical morphisms. Apply this notation to $\calC_\infty$.
These spaces are nonzero precisely when $y=x+1$, and
$\rad^2\calC_\infty(x,x+1)=0$ for all $x\in\ZZ$.
Since $\Phi$ is an automorphism, it induces an isomorphism
\[
 \overline\Phi:\operatorname{Irr}_{\calC_\infty}(i,i+1)
 \xrightarrow{\sim}
 \operatorname{Irr}_{\calC_\infty}(\Phi(i),\Phi(i+1))
\]
for all $i\in\ZZ$.
The left-hand side is nonzero, so the right-hand side is also
nonzero. Therefore $\Phi(i+1)=\Phi(i)+1$.
Putting $j:=\Phi(i)$, we obtain
\[
 \Phi\bigl(\calC_\infty(i,i+1)\bigr)
 =\calC_\infty(j,j+1).
\]
Applying $F\Phi=Q$ to this morphism space gives
\[
 \kk(a+s)+\kk b=\kk a+\kk b,
\]
contrary to the linear independence of $a,b,s$.
Thus every connected Galois covering $F:\calC\to A$
satisfies $F\Phi=\phi Q_\ell$ for some isomorphism
$\Phi:\calC_\ell\to\calC$ and some automorphism $\phi$ of $A$,
but the stronger requirement $F\Phi=Q_\ell$ need not hold
for any such $\Phi$.
\end{remark}

\section{Nonstandardness}\label{sec:nonstandardness}

We now show that the category $\calC_n$ is not simply connected
for all  $n \in \{1,2,3,\dots, \infty\}$.
For every integer $n\ge1$, the quiver of $\calC_n$ has vertices
$i\in\ZZ/n\ZZ$ and two parallel arrows from $i$ to $i+1$
for all $i\in\ZZ/n\ZZ$.
It contains an oriented cycle and is therefore not acyclic.
This includes $n=1$, where the arrows are loops,
and $n=2$, where arrows run in both directions between the two vertices.
Thus $\calC_n$ is not simply connected.
There is also a proper connected Galois covering
\begin{equation}\label{eq:finite-cover}
 \calC_\infty\longrightarrow\calC_n
\end{equation}
with group $\langle\delta^n\rangle$.

For $\calC_\infty$, use its repetitive presentation
\eqref{eq:cd-products}, whose defining relations, for all $i\in\ZZ$, are
\begin{equation}\label{eq:weight-relations}
 c_{i+1}c_i=0,\qquad d_{i+1}d_i=0,\qquad
 c_{i+1}d_i=d_{i+1}c_i.
\end{equation}

\begin{proposition}\label{prop:chain-not-sc}
The category $\calC_\infty$ is not simply connected.
\end{proposition}
\begin{proof}
{For all $i\in\ZZ$ and all arrows $f$, assign weights by}
\[
 \wt(c_i)=0,\qquad\wt(d_i)=1,\qquad
 \wt(f^{-1})=-\wt(f).
\]
For a walk $w=f_m\cdots f_1$, where each $f_i$ is an arrow
or a formal inverse arrow, define its \emph{weight} by
\[
 \wt(w):=\sum_{i=1}^{m}\wt(f_i).
\]
The constant walk has weight $0$.
We show that it is invariant under homotopy.

Replacing $f^{-1}f$ or $ff^{-1}$ by the constant walk
at the corresponding vertex does not change the weight,
since $\wt(f^{-1})=-\wt(f)$ for every arrow $f$.
Both paths in the binomial relation in
\eqref{eq:weight-relations} have weight $1$ for all $i\in\ZZ$.
The remaining defining relations are monomial and hence
homogeneous. Thus the relation ideal is generated by
homogeneous elements and is therefore homogeneous
for the weighting.
Let $I_\infty$ denote the ideal generated by
\eqref{eq:weight-relations}.
For a minimal relation $r=\sum_{j=1}^{m}\lambda_jp_j$,
where the $p_j$ are distinct paths and
$\lambda_j\in\kk^\times$ for all $j = 1,\dots, m$, put
\[
 r_d:=\sum_{\wt(p_j)=d}\lambda_jp_j
 \qquad\text{for all }d\in\ZZ.
\]
Since $I_\infty$ is homogeneous for the weighting,
$r_d\in I_\infty$ for all $d\in\ZZ$.
If paths of two or more different weights occurred in $r$,
each nonzero $r_d$ would be a nonempty proper subsum of $r$
belonging to $I_\infty$, contrary to minimality.
Thus all paths occurring in $r$ have the same weight.
Hence every generating homotopy relation preserves weight, as desired.

Now consider the closed walk
\begin{equation}\label{eq:weight-loop}
 \begin{tikzcd}[column sep=huge]
 i\arrow[r,bend left=20,"d_i"]&
 i+1\arrow[l,bend left=20,"c_i^{-1}"]
 \end{tikzcd}
 \qquad c_i^{-1}d_i.
\end{equation}
Its weight is $1$, whereas the constant walk has weight $0$.
Thus it is not null-homotopic.
This presentation has nontrivial fundamental group, so
$\calC_\infty$ is not simply connected.
\end{proof}

\begin{corollary}\label{cor:nonstandard}
The algebra $A$ has no simply connected Galois covering. In particular,
it is not standard.
\end{corollary}
\begin{proof}
Let $F:\calC\to A$ be a connected Galois covering.
By Proposition~\ref{prop:classification}, its source is isomorphic to
one of $\calC_n$ or $\calC_\infty$.
The finite cases are not acyclic, and the infinite case is excluded
by Proposition~\ref{prop:chain-not-sc}.
Thus no connected Galois source is simply connected.

For a Galois covering $\calR\to A$ with possibly disconnected
source and Galois group $G$, choose a connected component
$\calR_0$ and put $G_0=\{g\in G\mid g\calR_0=\calR_0\}$.
The group $G$ is transitive on $\Ob\calR$, since the quotient has
one object.
Let $g\in G$. If $gx\in\Ob\calR_0$ for some
$x\in\Ob\calR_0$, then the connected components
$g\calR_0$ and $\calR_0$ have the object $gx$ in common.
Thus $g\calR_0=\calR_0$, so $g\in G_0$.
Hence $G_0$ is transitive on $\Ob\calR_0$ because so is $G$ on $\Ob\calR$.
There are no morphisms between distinct components, so restriction
gives a connected Galois covering $\calR_0\to A$.
Thus allowing disconnected Galois sources whose connected components
are simply connected gives no additional possibility.
\end{proof}

\section{A simply connected covering and its composite}\label{sec:composite}

The preceding proof considered every Galois covering of $A$.
We now give the covering in {Theorem~A(\ref{item:A-composite})}. It shows directly that a
composite of two Galois coverings need not be Galois in this example.

Give $\calC_\infty$ the $\ZZ$-grading
$\deg c_i=0$, $\deg d_i=-1$ for all $i\in\ZZ$, and set
$\calD=\calC_\infty\#\ZZ$.
The relations \eqref{eq:weight-relations} are homogeneous.
Its objects are $p^{(q)}=(p,q)$ for all $p,q\in\ZZ$.
By \eqref{eq:smash-convention}, for all $p,q,r,s\in\ZZ$,
\[
 \calD(p^{(q)},r^{(s)})=(\calC_\infty)_{q-s}(p,r).
\]
Thus $c_p$ lifts from $p^{(q)}$ to $(p+1)^{(q)}$, and $d_p$
lifts from $p^{(q)}$ to $(p+1)^{(q+1)}$.
In the target-indexed notation of \cite[Definition~6.2.12]{AsaBook},
these arrows are $c_p^{(q)}$ and $d_p^{(q+1)}$.
For all $p,q\in\ZZ$, write
\[
 \begin{aligned}
 c_{p,q}:=c_p^{(q)}&:(p,q)\longrightarrow(p+1,q),\\
 d_{p,q}:=d_p^{(q+1)}&:(p,q)\longrightarrow(p+1,q+1).
 \end{aligned}
\]
The weighted-quiver calculation \cite[Theorem~6.2.18]{AsaBook}
gives the defining relations
\begin{align}
 c_{p+1,q}c_{p,q}&=0,&d_{p+1,q+1}d_{p,q}&=0,\label{eq:grid-relations}\\
 d_{p+1,q}c_{p,q}&=c_{p+1,q+1}d_{p,q}.&&\nonumber
\end{align}
The arrows and relations are indexed by all $(p,q)\in\ZZ^2$.
With equal horizontal and vertical coordinate units, the $c$ arrows
are horizontal and the $d$ arrows point up and to the right at
$45$ degrees. Each elementary parallelogram commutes:
\begin{equation}\label{eq:grid-square}
\begin{tikzcd}[column sep=50pt]
&(p+1,q+1) & (p+2,q+1)\\
(p,q) & (p+1,q)
\Ar{1-2}{1-3}{"{c_{p+1,q+1}}"}
\Ar{2-1}{2-2}{"{{c_{p,q}}}"'}
\Ar{2-1}{1-2}{"{d_{p,q}}"}
\Ar{2-2}{1-3}{"d_{p+1,q}"'}
\end{tikzcd}
\end{equation}
Every length-three path either contains two consecutive arrows of
the same kind or becomes such a path after one interchange using
\eqref{eq:grid-square}.
The zero relations therefore make every length-three path zero.

\begin{proposition}\label{prop:grid}
The category $\calD$ is locally bounded and simply connected.
The assignments
\begin{equation}\label{eq:P}
 P(p,q)=p,\qquad
 P(c_{p,q})=c_p,\qquad P(d_{p,q})=d_p
\end{equation}
for all $(p,q)\in\ZZ^2$ define a proper connected Galois covering
$P:\calD\to\calC_\infty$.
\end{proposition}
\begin{proof}
We divide the proof into three claims.

\smallskip\noindent
\textit{Claim 1. The category $\calD$ is locally bounded and connected.}
Every arrow increases $p$, so distinct vertices are nonisomorphic
and every endomorphism algebra is $\kk$.
For each vertex $z=(p,q)$, the direct sum
$\bigoplus_{w\in\Ob\calD}\calD(z,w)$ has a basis consisting of
$\id_z$, $c_{p,q}$, $d_{p,q}$ and the common value of the two paths
in \eqref{eq:grid-square}. Its dimension is $4$.
The source-varying direct sum $\bigoplus_{w\in\Ob\calD}\calD(w,z)$
likewise has dimension $4$.
Hence $\calD$ is locally bounded.
The horizontal and diagonal edges connect all vertices in $\ZZ^2$
as an unoriented graph: their displacement vectors $(1,0)$ and
$(1,1)$ generate $\ZZ^2$.
Hence $\calD$ is connected.

\smallskip\noindent
\textit{Claim 2. The functor $P$ is a proper Galois covering.}
In the smash-product coordinates, \eqref{eq:P} is the canonical
forgetful functor $\calC_\infty\#\ZZ\to\calC_\infty$.
It is therefore a Galois covering by
\cite[Proposition~5.8]{Asa}.
For all $(p,q)\in\ZZ^2$, the canonical action of $1\in\ZZ$ is
\[
 \tau(p,q)=(p,q+1),\qquad
 \tau(c_{p,q})=c_{p,q+1},\qquad
 \tau(d_{p,q})=d_{p,q+1}.
\]
For each $i\in\ZZ$, the fiber
\[
 P^{-1}(i)=\{(i,q)\mid q\in\ZZ\}
\]
is a single $\langle\tau\rangle$-orbit. Indeed,
$\tau^n(i,q)=(i,q+n)$ for all $n,q\in\ZZ$.
Thus $\calD/\langle\tau\rangle\cong\calC_\infty$, and $P$ is
proper because every object fiber is infinite.

\smallskip\noindent
\textit{Claim 3. Every bound-quiver presentation of $\calD$ has trivial
fundamental group.}
For each $(p,q)\in\ZZ^2$, put
\[
 u_{p,q}:=d_{p+1,q}c_{p,q},\qquad
 v_{p,q}:=c_{p+1,q+1}d_{p,q}.
\]
These are the two paths of length 2 from $(p,q)$ to
$(p+2,q+1)$ along the boundary of the elementary
parallelogram in \eqref{eq:grid-square}.
The binomial relation $u_{p,q}-v_{p,q}$ therefore gives
\[
 u_{p,q}\sim v_{p,q}
\]
in the bound-quiver homotopy relation.
These parallelograms tile the plane.
Every finite closed edge walk is therefore reduced to the constant
walk by parallelogram moves and immediate cancellations.
Thus the fundamental group of this bound-quiver presentation is trivial.
The monomial zero relations do not change these homotopies.

We must prove the assertion for every presentation.
Let $\Gamma$ be the displayed quiver, and let
$\pi':\kk\Gamma'\to\calD$ be any admissible presentation
with kernel $I'$.
Both $\Gamma$ and $\Gamma'$ are Gabriel quivers of $\calD$.
Thus there is a quiver isomorphism $f:\Gamma'\to\Gamma$
which sends each vertex $z$ to the object $\pi'(z)$
and each arrow to the unique arrow with the corresponding
source and target.

Every arrow of $\Ga$ increases $p$ by $1$, whereas every
path of length at least $2$ increases it by at least $2$.
Hence there is no path of length at least $2$ parallel
to an arrow.
Since each horizontal or diagonal edge
has exactly one arrow, for every arrow $\alpha'$ of $\Ga'$
there is a scalar $\lambda_{\alpha'}\in\kk^\times$ such that
\[
 \pi'(\alpha')=\lambda_{\alpha'}f(\alpha'),
\]
where $f(\alpha')$ also denotes its image in $\calD$.

Let $u',v'$ be the two paths of length $2$ in $\Ga'$
whose images under $f$ form an elementary parallelogram.
By the preceding formula for the images of arrows,
$\pi'(u')$ and $\pi'(v')$ are nonzero scalar multiples
of the nonzero images of $f(u')$ and $f(v')$, respectively.
They belong to the one-dimensional space
$\calD((p,q),(p+2,q+1))$,
where $(p,q)$ is their common source under $f$.
Hence $\pi'(u')=\lambda\pi'(v')$ for some
$\lambda\in\kk^\times$, and thus $u'-\lambda v'\in I'$.
This relation is minimal, since neither $u'$ nor $v'$
belongs to $I'$, and therefore gives $u'\sim v'$.
Consequently, as for the displayed presentation,
every finite closed walk contracts by parallelogram
moves and immediate cancellations.
Hence the fundamental group of $(\Ga',I')$ is trivial.
Since $\Ga'\cong \Ga$ is acyclic, this proves that
$\calD$ is simply connected.
\end{proof}

Set $H=QP$. 
A composite of covering functors is a covering functor.
Indeed, fix $x\in\Ob\calD$ and $a\in\Ob A$.
Since $Q$ and $P$ are covering functors, respectively,
we have isomorphisms
\[
 A(QPx,a)
 \cong
 \bigoplus_{\substack{y\in\Ob\calC_\infty\\Qy=a}}
 \calC_\infty(Px,y)
 \cong
 \bigoplus_{\substack{y\in\Ob\calC_\infty\\Qy=a}}
 \bigoplus_{\substack{z\in\Ob\calD\\Pz=y}}
 \calD(x,z)
 =
 \bigoplus_{\substack{z\in\Ob\calD\\QPz=a}}
 \calD(x,z).
\]
The inverse of this composite isomorphism sends
$(f_z)_z$ to $\sum_z QP(f_z)$.
Thus it is the target-varying covering map for $QP$
in \eqref{eq:dfn-covering1}.
The source-varying covering map in \eqref{eq:dfn-covering2}
is an isomorphism by the same argument with targets fixed.
Since $P$ and $Q$ are surjective on objects,
so is $QP$. Hence $QP$ is a covering functor.
Its values on arrows are
\begin{equation}\label{eq:H}
 H(c_{p,q})=(-1)^p a,\qquad
 H(d_{p,q})=b-p a.
\end{equation}

{These formulas hold for all $(p,q)\in\ZZ^2$ and follow by applying
$Q$ to \eqref{eq:basis}.}

\begin{proposition}\label{prop:composite}
The covering $H:\calD\to A$ is not Galois.
\end{proposition}

\begin{proof}
Let $G$ be the deck group of $H$, and write $\bar H$ for the
induced maps $\operatorname{Irr}_{\calD}(x,y)\to V$
for all $x,y\in\Ob\calD$.
By \eqref{eq:grid-square} and \eqref{eq:H}, for every $z=(p,q)$,
\[
 \begin{aligned}
 \bar H\operatorname{Irr}_{\calD}(z,z+(1,0))
   &=\kk\bar a,\\
 \bar H\operatorname{Irr}_{\calD}(z,z+(1,1))
   &=\kk(\bar b-p\bar a).
 \end{aligned}
\]
Let $T\in G$ and write $Tz=(p',q')$.
Since $T$ permutes the outgoing irreducible spaces, $HT=H$ gives
\begin{equation}\label{eq:line-pair}
 \{\kk\bar a,\kk(\bar b-p\bar a)\}
 =\{\kk\bar a,\kk(\bar b-p'\bar a)\}.
\end{equation}
The linear independence of $\bar a,\bar b$ excludes
$\kk(\bar b-p\bar a)=\kk\bar a$. Hence
\[
 \bar b-p\bar a=t(\bar b-p'\bar a)
 \quad\text{for some }t\in\kk^\times,
 \qquad t=1,\quad p=p'\text{ in }\kk.
\]
Consequently,
\[
 Gz\subseteq
 \{r\in\ZZ\mid r1_{\kk}=p1_{\kk}\}\times\ZZ
 \subsetneq\ZZ^2=H^{-1}(*).
\]
Thus $G$ is not transitive on this fiber, and $H$ is not Galois.
\end{proof}

One can also see directly why $\delta$ does not lift through $P$.
Such a lift would make the following square commute:
\begin{equation}\label{eq:no-lift-square}
\begin{tikzcd}[column sep=huge]
 \calD\arrow[r,"T"]\arrow[d,"P"']&\calD\arrow[d,"P"]\\
 \calC_\infty\arrow[r,"\delta"']&\calC_\infty.
\end{tikzcd}
\end{equation}
Choose the diagonal arrow $d_{p,q}$.
If $PT=\delta P$, then $PT(p,q)=p+1$.
An automorphism $T$ sends the line of $d_{p,q}$ to one of the
horizontal or diagonal irreducible arrow lines at $T(p,q)$.
Consequently $PT(d_{p,q})$ lies in $\kk c_{p+1}$ or $\kk d_{p+1}$.
But \eqref{eq:mixing} gives
$\delta P(d_{p,q})=d_{p+1}+(-1)^{p+1}c_{p+1}$, which lies in
neither line. Hence no such lift $T$ exists.

\begin{remark}\label{rem:characteristic}
The constructions of $\calC_\infty$, $\calD$, $P$ and $H=QP$
also apply in positive odd characteristic.
Let $\ell=\operatorname{char}\kk>0$.
The coefficients $i$ in \eqref{eq:basis} are periodic, since
$(i+\ell)1_{\kk}=i1_{\kk}$ for all $i\in\ZZ$.
However, the objects $i$ and $i+\ell$ are distinct, and
$\delta^n(i)=i+n\ne i$ for all $i\in\ZZ$ and all positive
integers $n$. Hence $\delta$ still has infinite order.
Since $(p+1)1_{\kk}\ne p1_{\kk}$ for all $p\in\ZZ$,
the strict inclusion in the proof of
Proposition~\ref{prop:composite} holds also in positive characteristic.
The assumption $\operatorname{char}\kk\ne2$ is used in
Lemma~\ref{lem:one-degree}, where the matrix in \eqref{eq:matrix} must
be nonsymmetric.
\end{remark}

\section{A sufficient compatibility hypothesis for standardness}\label{sec:compatibility}

The example suggests supplementing admissibility on objects by a
condition on lifting through a Galois covering with simply connected
source. Preservation of the chosen arrow lines of a presentation
alone is insufficient: $\delta$ permutes the original family of
arrow lines, sending $\kk a_i$ to $\kk a_{i+1}$ and $\kk b_i$ to
$\kk b_{i+1}$ for every $i\in\ZZ$, while its orbit algebra $A$ is
not standard. 
Theorem~\ref{thm:repair-lifts} below states a sufficient lifting
condition for standardness.

All actions and functor equalities in this section are strict in the
sense specified in Section~\ref{sec:conventions}.

\begin{lemma}
\label{lem:covering-uniqueness}
The following statements hold.
\begin{enumerate}
\item
Let $F:\calC\to\calR$ and $L_1,L_2:\calD\to\calC$
be coverings, with $\calD$ connected. If
\[
 FL_1=FL_2,\qquad L_1z=L_2z
 \quad\text{for some }z\in\Ob\calD,
\]
then $L_1=L_2$.
\item
Let $F:\calD\to\calR$ be a connected Galois covering
with group $G$, where $\calD$ is nonempty, and let
$\rho:G\to\operatorname{Aut}(\calD)$ be the action of $G$.
Then
\[
 \operatorname{Deck}(F)
 :=\{T\in\operatorname{Aut}(\calD)\mid FT=F\}
 =\rho(G),
\]
and $\rho$ induces an isomorphism
\[
 G\xrightarrow{\ \sim\ }\operatorname{Deck}(F).
\]
\end{enumerate}
\end{lemma}

\begin{proof}
(1) Suppose that $L_1x=L_2x=a$, and let $0\ne f:x\to y$.
The target-varying covering map
\[
 F_*:
 \bigoplus_{Fb=FL_1y}\calC(a,b)
 \xrightarrow{\sim}\calR(Fa,FL_1y)
\]
gives
\[
 F_*(\iota_1L_1(f))=FL_1(f)=FL_2(f)
 =F_*(\iota_2L_2(f)),
\]
where $\iota_i$ is the inclusion of the summand indexed by
$L_i y$, for $i=1,2$. 
Since $F_*$ is bijective, this shows $\iota_1L_1(f)=\iota_2L_2(f)$.
Since $L_1$ is faithful, and $\iota_1$ is a monomorphism, we have
$\iota_1L_1(f) \ne 0$.  As a consequence,
\[
 0\ne\iota_1L_1(f)=\iota_2L_2(f).
\]
If $L_1y \ne L_2y$, then $\iota_1L_1(f)=\iota_2L_2(f) = 0$
because the two elements lie in distinct summands
of the domain of $F_*$ , a contradiction.
Thus $L_1y=L_2y$.
For $0\ne f:y\to x$, use the source-varying covering map.
Since $\calD$ is connected, this gives $L_1x=L_2x$ for all
$x\in\Ob\calD$. Injectivity of $F$ on each morphism space
then gives $L_1(f)=L_2(f)$ for every morphism $f$.

(2) By assumption, we have
\[
 F\rho(g)=F\quad(g\in G),\qquad F^{-1}(Fz)=Gz
 \quad(z\in\Ob\calD).
\]
Let $T\in\operatorname{Deck}(F)$ and fix $z\in\Ob\calD$.
There exists $g\in G$ such that $Tz=\rho(g)z$. Since
\[
 FT=F=F\rho(g),
\]
(1) gives $T=\rho(g)$. Thus $\operatorname{Deck}(F)=\rho(G)$.

Since $\calD$ is nonempty and the $G$-action is free on
objects, $\rho$ is injective. Therefore it induces an
isomorphism $G\xrightarrow{\sim}\operatorname{Deck}(F)$.
\end{proof}

\begin{theorem}\label{thm:repair-lifts}
Let $P:\calD\to\calR$ be a connected Galois covering with
nonempty source and group $H$, and let
$\rho_H:H\to\operatorname{Aut}(\calD)$ be the action of $H$.
Let $G$ act freely on the objects of $\calR$, and let
$Q:\calR\to\calR/G$ be the classical orbit projection.
Then $QP$ is Galois if and only if, for all $g\in G$, there is
an automorphism $\widetilde g$ of $\calD$ such that the following
diagram commutes:
\[
\begin{tikzcd}[column sep=huge]
 \calD\arrow[r,"\widetilde g"]\arrow[d,"P"']&
 \calD\arrow[d,"P"]\\
 \calR\arrow[r,"g"']&\calR
\end{tikzcd}
\]
When these conditions hold, the Galois group of $QP$ is
\[
 E=\{(g,T)\in G\times\operatorname{Aut}(\calD)\mid PT=gP\},
\]
acting on $\calD$ by
\begin{equation}
\label{eq:action-E}
 (g,T)(z)=Tz,\qquad (g,T)(f)=T(f)
\end{equation}
for all $(g,T)\in E$, objects $z$ and morphisms $f$ of $\calD$.
Moreover, the sequence
\[
 \begin{gathered}
 1\longrightarrow H\xrightarrow{\ \iota\ }E
 \xrightarrow{\ \pi\ }G\longrightarrow1,\\
 \iota(h)=(1,\rho_H(h))\quad(h\in H),\qquad
 \pi(g,T)=g\quad((g,T)\in E)
 \end{gathered}
\]
is exact, and the action of $E$ induces an isomorphism
\[
 E\xrightarrow{\ \sim\ }\operatorname{Deck}(QP),
 \qquad (g,T)\longmapsto T.
\]
The composite $QP$ is an $E$-Galois covering.
Under these equivalent lifting conditions, the $E$-orbits in
$\Ob\calD$ correspond via $P$ to the $G$-orbits in $\Ob\calR$.
In particular, if the latter set of orbits is finite and $\calD$
is simply connected, the algebra associated with the orbit
category $\calR/G$ is standard.
\end{theorem}
\begin{proof}

Componentwise composition makes $E$ a group.
By Lemma~\ref{lem:covering-uniqueness}(2),
$\operatorname{Deck}(P)=\rho_H(H)$ and $\rho_H$ is injective.
Hence $\iota$ is injective, and
\[
 \Ker\pi
 =\{(1,T)\mid PT=P\}
 =\{1\}\times\operatorname{Deck}(P)
 =\{1\}\times\rho_H(H)
 =\iota(H).
\]
Put $F=QP$. 

($\Leftarrow$)
Assume that every $g\in G$ lifts through $P$. Then $\pi$ is
surjective, so the asserted sequence is exact.

By the definition of a Galois covering with group $E$, it is
enough to check the following:
\begin{enumerate}
\item $E$ acts freely and locally boundedly on $\calD$;
\item $F$ is a covering functor;
\item $F$ is strictly $E$-invariant;
\item $F$ is surjective on objects;
\item $E$ acts transitively on each object fiber of $F$.
\end{enumerate}

\noindent
(1) Suppose that $(g,T)z=z$ for some $(g,T)\in E$
and $z\in\Ob\calD$. Then $Tz=z$ by \eqref{eq:action-E}, and
\[
 gPz=PTz=Pz.
\]
Since $G$ acts freely on $\Ob\calR$, we have $g=1$.
Hence $PT=P$, and therefore
\[
 T\in\operatorname{Deck}(P)=\rho_H(H).
\]
Choose $h\in H$ such that $T=\rho_H(h)$. Since $Tz=z$ and
$H$ acts freely on $\Ob\calD$, we have $h=1$, and hence
$T=1$. Thus $(g,T)=1_E$, so $E$ acts freely on the objects
of $\calD$.
Since $\calD$ is locally bounded, the $E$-action is also
locally bounded by the observation in
Section~\ref{sec:conventions}.

\noindent
(2) By the composition argument preceding
Proposition~\ref{prop:composite}, $F$ is a covering.

\noindent
(3) For every $(g,T)\in E$, we have
\[
 F(g,T) \overset{\eqref{eq:action-E}}{=} FT=QPT=QgP=QP=F.
\]
Thus $F$ is strictly $E$-invariant.

\noindent
(4) Since $P$ and $Q$ are surjective on objects, so is $F=QP$.

\noindent
(5) Let $z,w\in\Ob\calD$ with $Fz=Fw$. 
Then $Q(Pz) = Q(Pw)$.
Since $Q$ is a Galois covering with group $G$, we have $Pw=gPz$
for some $g\in G$. Choose a lift $T$ of $g$. Then $Pw=gPz=P(Tz)$.
Since $P$ is a Galois covering with group $H$,
\[
 w=\rho_H(h)Tz \overset{\eqref{eq:action-E}}{=} (g,\rho_H(h)T)z
 \qquad\text{for some }h\in H \ya{\rho_H} \operatorname{Deck}(P),
\]
where $P\rho_H(h)T = PT = gP$, and hence $(g, \rho_H(h)T) \in E$.
Thus $E$ acts transitively on each object fiber of $F$.

Consequently, $F=QP$ is a Galois covering with group $E$.
By Lemma~\ref{lem:covering-uniqueness}(2), the action of
$E$ therefore induces an isomorphism
\[
 E\xrightarrow{\ \sim\ }\operatorname{Deck}(F),
 \qquad (g,T)\longmapsto T.
\]
Moreover, its object fibers give the bijection
\[
 (\Ob\calD)/E\xrightarrow{\ \sim\ }(\Ob\calR)/G,
 \qquad Ez\longmapsto G(Pz).
\]

($\Rightarrow$).
Conversely, suppose that $F$ is Galois.
Let $g\in G$.
Since $\Ob\calD \ne \emptyset$, we can take $z\in\Ob\calD$.
Choose any $w \in P^{-1}(gPz)$.
Then $Pw = gPz$, and $Fw = QPw = (Qg)Pz = QPz = Fz$.
Hence $Tz=w$ for some $T \in \Deck(F)$. Then
\[
 (PT)z=Pw=gPz=(gP)z,
 \qquad
 Q(PT) = QP = Q(gP).
\]
Applying Lemma~\ref{lem:covering-uniqueness}(1) to the coverings
$Q$, $PT$ and $gP$ gives $PT=gP$.
Thus every $g\in G$ lifts through $P$.

Finally, if $(\Ob\calR)/G$ is finite and $\calD$ is simply
connected, the Galois covering $F:\calD\to\calR/G$
gives the asserted standardness by definition.
\end{proof}

\begin{remark}
\label{rem:repair-coherence}
Theorem~\ref{thm:repair-lifts} does not require chosen lifts to
satisfy $\widetilde{gh}=\widetilde g\widetilde h$ for all $g,h\in G$:
all lifts together form $E$.
The extension need not split for a general group $G$.
For the infinite cyclic group in the Proposition of \cite[\S2.1]{Sko},
it is enough to lift a generator $g$.
Any chosen lift $T$ gives a splitting $g^n\mapsto(g^n,T^n)$, so that
$E\cong H\rtimes\ZZ$, where $1\in\ZZ$ acts on $H$ by
$h\mapsto ThT^{-1}$.
\end{remark}

\subsection*{Compatible arrow lines for the repetitive Kronecker category}

For all $i\in\ZZ$, write $V_i=\calC_\infty(i,i+1)$ and
$W_i=\calC_\infty(i,i+2)$. Composition gives a nondegenerate pairing
\[
 \mu_i:V_{i+1}\times V_i\longrightarrow W_i,\qquad (v,u)\longmapsto vu.
\]
Here {\emph{nondegenerate}} means that a vector which pairs to zero with every
vector in the other space is zero. Let $\mathbb P(V_i)$ denote the
set of lines in $V_i$. Define the {\emph{zero-product transport}}
\[
 \begin{split}
 Z_i:\mathbb P(V_i)&\longrightarrow\mathbb P(V_{i+1}),\\
 Z_i(L)&=\{v\in V_{i+1}\mid\mu_i(v,\ell)=0
                      \text{ for all }\ell\in L\}.
 \end{split}
\]
For every $i\in\ZZ$, nondegeneracy makes the linear map
\[
 V_{i+1}\longrightarrow\operatorname{Hom}_{\kk}(V_i,W_i),
 \qquad v\longmapsto\mu_i(v,\text{-})
\]
injective. Since $\dim_{\kk}V_i=\dim_{\kk}V_{i+1}=2$ and
$\dim_{\kk}W_i=1$, we have
\[
 \dim_{\kk}V_{i+1}=2=2\cdot1
 =\dim_{\kk}\operatorname{Hom}_{\kk}(V_i,W_i),
\]
so this map is an isomorphism. Under this isomorphism, the
annihilator $Z_i(L)$ of any line $L\subseteq V_i$ corresponds to
$\operatorname{Hom}_{\kk}(V_i/L,W_i)$, of dimension $(2-1)\cdot1=1$.
Thus $Z_i(L)$ is a line in $V_{i+1}$.

In repetitive bases $(c_i,d_i)_{i\in\ZZ}$, the relations
\eqref{eq:cd-products} give, for every $i\in\ZZ$,
\[
 Z_i\bigl(\kk(xc_i+yd_i)\bigr)
 =\kk(xc_{i+1}-yd_{i+1})
 \quad\text{for all }(x,y)\in\kk^2\setminus\{(0,0)\}.
\]
Hence $Z_i$ is induced by the linear isomorphism
\[
 V_i\xrightarrow{\sim}V_{i+1},\qquad
 c_i\longmapsto c_{i+1},\quad d_i\longmapsto-d_{i+1}.
\]
Its matrix with respect to the ordered bases $(c_i,d_i)$ and
$(c_{i+1},d_{i+1})$ is
$\left(\begin{smallmatrix}1&0\\0&-1\end{smallmatrix}\right)$.
A \emph{compatible pair system} is a family
$\mathcal S=(\mathcal S_i)_{i\in\ZZ}$, where
\[
 \mathcal S_i=\{L_i,M_i\},\qquad
 L_i,M_i\in\mathbb P(V_i),\quad L_i\ne M_i
 \quad\text{for all }i\in\ZZ,
\]
such that
\[
 L_{i+1}=Z_i(L_i),\qquad
 M_{i+1}=Z_i(M_i)
 \quad\text{for all }i\in\ZZ.
\]

Note that compatibility is defined by composition without choosing bases
of the spaces $V_i$.

Let $g\in\operatorname{Aut}(\calC_\infty)$.
For every $i\in\ZZ$, $g$ induces an isomorphism
\[
 0\ne\operatorname{Irr}_{\calC_\infty}(i,i+1)
 \xrightarrow[\sim]{\overline g}
 \operatorname{Irr}_{\calC_\infty}(g(i),g(i+1))\ne0.
\]
The irreducible morphism spaces are nonzero precisely between
successive objects. Hence
\[
 g(i+1)=g(i)+1\qquad\text{for all }i\in\ZZ.
\]
Putting $r=g(0)$, we obtain $g(i)=i+r$ for all $i\in\ZZ$.
For any linear isomorphism $f:V\to V'$ and any unordered pair
$\{L,M\}$ of lines in $V$ with $L\ne M$, we write
\[
 f(\{L,M\}):=\{f(L),f(M)\}.
\]
We say that $g$ {\emph{preserves the system}} if
\[
g(\mathcal S_i)=\mathcal S_{i+r}\qquad\text{for all }i\in\ZZ.
\]

\subsubsection*{The grid covering associated with a compatible pair system.}
Fix a compatible pair system $\mathcal S=(\{L_i,M_i\})_{i\in\ZZ}$.
We first construct a covering depending on this system and on choices
of generators, but not on an automorphism of $\calC_\infty$.
For all $i\in\ZZ$, choose $0\ne u_i\in L_i$ and
$0\ne d'_i\in M_i$. Compatibility gives
$u_{i+1}u_i=d'_{i+1}d'_i=0$.
The mixed products $d'_{i+1}u_i$ and $u_{i+1}d'_i$ are nonzero:
otherwise nondegeneracy of $\mu_i$ would be contradicted, since
$L_i\ne M_i$ and $L_{i+1}\ne M_{i+1}$.
As $\dim_{\kk}W_i=1$, there is a unique $\alpha_i\in\k^\times$ with
\[
 d'_{i+1}u_i=\alpha_i u_{i+1}d'_i.
\]
Starting with $t_0=1$, define $t_i\in\k^\times$ in both directions by
\[
 \begin{aligned}
 t_{i+1}&=\alpha_i t_i&&\text{for all integers }i\ge0,\\
 t_i&=\alpha_i^{-1}t_{i+1}&&\text{for all integers }i<0,
 \end{aligned}
\]
and put $c'_i=t_i u_i$ for all $i\in\ZZ$.
Then, for every $i\in\ZZ$,
\[
 c'_{i+1}c'_i=d'_{i+1}d'_i=0,\qquad
 d'_{i+1}c'_i=c'_{i+1}d'_i\ne0.
\]
These generators give exactly the presentation
\eqref{eq:weight-relations} of $\calC_\infty$.
Indeed, they form a basis of each $V_i$, and their common mixed
product spans $W_i$. The displayed relations make all paths of
length three zero, so comparison of the morphism spaces shows
that there are no further relations.

Let $\calD$ be the grid category of Proposition~\ref{prop:grid},
with the arrows and relations in \eqref{eq:grid-relations}.
Define
\[
 \begin{split}
 P_{\mathcal S}:\calD&\longrightarrow\calC_\infty,\\
 P_{\mathcal S}(p,q)&=p,\\
 P_{\mathcal S}(c_{p,q})&=c'_p,\qquad
 P_{\mathcal S}(d_{p,q})=d'_p
 \end{split}
 \qquad\text{for all }(p,q)\in\ZZ^2.
\]
The relations just verified make this a well-defined functor.
Give $\calC_\infty$ the grading
$\deg c'_i=0$, $\deg d'_i=-1$ for all $i\in\ZZ$.
The weighted-quiver calculation \cite[Theorem~6.2.18]{AsaBook},
with the smash-product convention \eqref{eq:smash-convention},
identifies $\calD$ with this smash product and $P_{\mathcal S}$
with its forgetful functor. Thus $P_{\mathcal S}$ is a connected
Galois covering with deck group $\langle\tau\rangle\cong\ZZ$,
where $\tau$ is the translation $(p,q)\mapsto(p,q+1)$ of
Proposition~\ref{prop:grid}.
Its source $\calD$ is locally bounded and simply connected by that
proposition. We call $P_{\mathcal S}$ the \emph{associated grid covering}
and fix the above choices once and for all.

\begin{proposition}
\label{prop:repair-line-system}
Let $\mathcal S$ be a compatible pair system of
$\calC_\infty\cong\widehat K$, and let
$P_{\mathcal S}:\calD\to\calC_\infty$ be the associated grid
covering constructed above.
For every $g\in\operatorname{Aut}(\calC_\infty)$, the following
conditions are equivalent:
\begin{enumerate}
\item\label{item:repair-preserves} $g$ preserves $\mathcal S$;
\item\label{item:repair-lifts} $g$ lifts through $P_{\mathcal S}$;
that is, there exists $\widetilde g\in\operatorname{Aut}(\calD)$ such that
\[
 P_{\mathcal S}\widetilde g=gP_{\mathcal S}.
\]
\end{enumerate}
Consequently, if an admissible group $G$ of automorphisms of
$\calC_\infty$ preserves $\mathcal S$, the composite
$\calD\xrightarrow{P_{\mathcal S}}\calC_\infty\to\calC_\infty/G$
is a Galois covering with simply connected source, and
$\calC_\infty/G$ is standard.
\end{proposition}

\begin{proof}
Put $P=P_{\mathcal S}$ and write $g(i)=i+r$ for all $i\in\ZZ$.
For $z=(i,q)\in\ZZ^2$ and $s\in\{0,1\}$, put
\[
 e_s=(1,s),\qquad
 (a_{z,0},a_{z,1})=(c_{i,q},d_{i,q}),\qquad
 (a'_{i,0},a'_{i,1})=(c'_i,d'_i).
\]
Then $a_{z,s}:z\to z+e_s$, $P(a_{z,s})=a'_{i,s}$, and
\[
 a'_{i+1,t}a'_{i,s}=0\quad\Longleftrightarrow\quad t=s
 \qquad(s,t\in\{0,1\}).
\]

\noindent
(\ref{item:repair-preserves}) $\Rightarrow$ (\ref{item:repair-lifts}).
Assume that $g$ preserves $\mathcal S$.
For each $i\in\ZZ$, there are a permutation $\sigma_i$ of $\{0,1\}$
and $\lambda_{i,s}\in\k^\times$ such that
\[
 g(a'_{i,s})=\lambda_{i,s}a'_{i+r,\sigma_i(s)},\qquad
 0=g(a'_{i+1,s}a'_{i,s})
 \ \Longrightarrow\ \sigma_{i+1}(s)=\sigma_i(s).
\]
Thus $\sigma_i=\sigma$ is independent of $i$.
Applying $g$ to $a'_{i+1,1}a'_{i,0}=a'_{i+1,0}a'_{i,1}\ne0$ gives
\[
 \kappa_i:=\lambda_{i+1,1}\lambda_{i,0}
          =\lambda_{i+1,0}\lambda_{i,1}.
\]
We define $T$ on objects and arrows by
\[
 T(i,q)=
 \begin{cases}
 (i+r,q)&\text{if $\sigma=\mathrm{id}$},\\
 (i+r,i-q)&\text{if $\sigma=(0\ 1)$},
 \end{cases}
 \qquad
 T(a_{z,s})=\lambda_{i,s}a_{Tz,\sigma(s)}.
\]
Since $T(z+e_s)=Tz+e_{\sigma(s)}$, these arrow maps have the
required sources and targets. The relations \eqref{eq:grid-relations}
are preserved because
\[
 \begin{aligned}
 T(a_{z+e_s,s}a_{z,s})&=0,\\
 T(a_{z+e_0,1}a_{z,0})
 &=\kappa_i a_{Tz+e_{\sigma(0)},\sigma(1)}a_{Tz,\sigma(0)}\\
 &=\kappa_i a_{Tz+e_{\sigma(1)},\sigma(0)}a_{Tz,\sigma(1)}
 =T(a_{z+e_1,0}a_{z,1}).
 \end{aligned}
\]
The assignments are bijective on objects and arrow lines, and
permute the defining relations up to nonzero scalars.
Hence $T\in\operatorname{Aut}(\calD)$. Moreover,
\[
 PT(z)=i+r=gP(z),\qquad
 PT(a_{z,s})=\lambda_{i,s}a'_{i+r,\sigma(s)}=gP(a_{z,s}),
\]
so $PT=gP$.

\noindent
(\ref{item:repair-lifts}) $\Rightarrow$ (\ref{item:repair-preserves}).
Let $T\in\operatorname{Aut}(\calD)$ satisfy $PT=gP$.
For each $z=(i,q)\in\ZZ^2$, put
\[
 \mathcal A_z:=\{\kk a_{z,0},\kk a_{z,1}\}.
\]
These are precisely the two nonzero irreducible morphism spaces
starting at $z$; here
$\operatorname{Irr}_{\calD}(z,z+e_s)=\kk a_{z,s}$ because no
path of length at least two is parallel to the path $a_{z,s}$ of length 1. Thus
$T(\mathcal A_z)=\mathcal A_{Tz}$.
Since $PT=gP$ gives $Tz=(i+r,q')$ for some $q'\in\ZZ$, we have
\[
 g(\mathcal S_i)=gP(\mathcal A_z)
 =PT(\mathcal A_z)=P(\mathcal A_{Tz})=\mathcal S_{i+r}.
\]
Thus $g$ preserves $\mathcal S$.

Finally, if $G$ preserves $\mathcal S$, every $g\in G$ lifts along
$P$. Theorem~\ref{thm:repair-lifts} gives the asserted Galois
covering, and simple connectedness of $\calD$ gives standardness.
\end{proof}

Let $G=\langle g\rangle$ be admissible and preserve $\mathcal S$.
Its generator shifts objects nontrivially, so $G\cong\ZZ$.
The composite $\calD\xrightarrow{P_{\mathcal S}}\calC_\infty
\to\calC_\infty/G$ is Galois with simply connected source.
For the lift $T$ above, the Galois group $E$ is given by
\[
 \begin{array}{lll}
 T\tau T^{-1}=\tau,&E\cong\ZZ^2,&\text{if $g$ preserves the labels},\\
 T\tau T^{-1}=\tau^{-1},&E\cong\ZZ\rtimes_{-1}\ZZ,
    &\text{if $g$ interchanges them}.
 \end{array}
\]
In the second case, a generator of the second factor acts on the
first by negation.

\subsection*{Why an arbitrary arrow presentation does not suffice}

In Proposition~\ref{prop:repair-line-system}, preservation of a
compatible pair system ensures that the automorphisms lift through
a Galois covering with simply connected source.
The following example shows that preservation of the arrow lines
of an arbitrary presentation does not suffice for this conclusion.
For the presentation \eqref{eq:chain-relations}, the minimal relations
give
\[
 a_{i+1}b_i\sim b_{i+1}a_i\sim b_{i+1}b_i.
\]
Cancellation in the second equivalence gives $a_i\sim b_i$.
Every walk consequently reduces to a walk in the integer line, whose
closed walks contract. Thus this presentation has trivial fundamental
group, while $\calC_\infty$ is not intrinsically simply
connected, as Proposition~\ref{prop:chain-not-sc} proves using a
different presentation.
The automorphism $\delta$ satisfies $\delta(a_i)=a_{i+1}$ and
$\delta(b_i)=b_{i+1}$ for all $i\in\ZZ$.
These arrow lines do not form a compatible pair system. Indeed,
\[
 Z_i(\kk a_i)=\kk a_{i+1},\qquad
 Z_i(\kk b_i)=\kk(b_{i+1}-a_{i+1})
 \ne\kk b_{i+1}
 \quad\text{for all }i\in\ZZ.
\]
Nevertheless, although $\delta$ preserves these arrow lines,
the orbit algebra
$\calC_\infty/\langle\delta\rangle$ is nonstandard by
Corollary~\ref{cor:nonstandard}.
Thus preservation of these arrow lines cannot replace preservation
of a compatible pair system in Proposition~\ref{prop:repair-line-system}.

\Needspace{18\baselineskip}
\begin{proposition}\label{prop:repair-monodromy}
Assume that $\kk$ is algebraically closed and
$\operatorname{char}\kk\ne2$.
Let $g\in\operatorname{Aut}(\calC_\infty)$ act on objects by
$g(i)=i+r$ {for all $i\in\ZZ$, where $r\in\ZZ\setminus\{0\}$}.
Choose any family of repetitive bases $(c_i,d_i)_{i\in\ZZ}$ in the
sense of Section~\ref{sec:repetitive}.
For all $i\in\ZZ$, let $M_i\in\operatorname{GL}_2(\kk)$ be the
matrix of $g\colon V_i\to V_{i+r}$ in these bases, using column
coordinates, and put
\[
 S=\begin{pmatrix}1&0\\0&-1\end{pmatrix},\qquad T_g=S^rM_0.
\]
Thus $S$ is the matrix, with respect to the original repetitive
bases, of the linear isomorphisms inducing the maps $Z_i$.
Then the following conditions are equivalent:
\begin{enumerate}
\item\label{item:monodromy-pairs} $g$ preserves a compatible pair system.
\item\label{item:monodromy-lift} $g$ lifts through a grid covering associated with some compatible
pair system.
\item\label{item:monodromy-diagonal} $T_g$ is diagonalizable over $\kk$.
\end{enumerate}
The last condition is independent of the choice of repetitive bases and
of the initial index. The action of $\langle g\rangle$ is admissible
because $r\ne0$ gives a free action with $|r|$ object orbits.
Under these equivalent conditions,
$\calC_\infty/\langle g\rangle$ is standard.
\end{proposition}

\begin{proof}
For every $i\in\ZZ$, put 
\begin{equation}
\label{eq:wi}
w_i=c_{i+1}d_i=d_{i+1}c_i
\end{equation}
and
write $g(w_i)=\rho_iw_{i+r}$ with $\rho_i\in\k^\times$.
Recall that $\mu_i\colon V_{i+1}\times V_i\to W_i$ is given
by $\mu_i(v,u)=vu$.
Identifying $W_i=\kk w_i$ with $\kk$ by $w_i\mapsto1$,
we regard $\mu_i$ as a scalar-valued bilinear form.
With respect to the ordered bases $(c_{i+1},d_{i+1})$ and
$(c_i,d_i)$, equalities \eqref{eq:cd-products} and \eqref{eq:wi} give the matrix $J_0$ of $\mu_i$ as follows:
\[
 J_0=\begin{pmatrix}0&1\\1&0\end{pmatrix}.
\]
We denote the transpose of a matrix $M$ by $M^{\mathsf t}$ and set $M^{-\sft}:= (M\inv)^\sft$.
For all column vectors $x,y\in\kk^2$, put
\[
 u=(c_i\ d_i)x,\qquad v=(c_{i+1}\ d_{i+1})y.
\]
Then
\[
 vu=(y^{\mathsf t}J_0x)w_i,\qquad
 g(u)=(c_{i+r}\ d_{i+r})M_ix,\qquad
 g(v)=(c_{i+r+1}\ d_{i+r+1})M_{i+1}y.
\]
Because $g(vu)=g(v)g(u)$, comparison of the coefficients of
$w_{i+r}$ gives
\[
 \rho_i y^{\mathsf t}J_0x
 =y^{\mathsf t}M_{i+1}^{\mathsf t}J_0M_ix
 \quad\text{for all }x,y\in\kk^2.
\]
Thus $M_{i+1}^{\mathsf t}J_0M_i=\rho_iJ_0$ for every $i\in\ZZ$.
Since $J_0^2 = \id$ and $J_0^\sft = J_0$, we have 
$$
M^{\sft}_{i+1} = \ro_iJ_0 M_i\inv J_0,\quad
\text{ and so }\quad M_{i+1} = \ro_iJ_0M_i^{-\sft} J_0.
$$
For every $M=\left(\begin{smallmatrix}a&b\\c&d\end{smallmatrix}\right)
\in\operatorname{GL}_2(\kk)$, we have
\[
 J_0M^{-\mathsf t}J_0
 =\frac1{\det M}\begin{pmatrix}a&-b\\-c&d\end{pmatrix}
 =(\det M)^{-1}SMS.
\]
Hence
\begin{equation}\label{eq:repair-projective-recursion}
 [M_{i+1}]=[SM_iS],\quad \text{ and so }\quad [M_i]=[S^iM_0S^i]
 \quad\text{for all }i\in\ZZ.
\end{equation}
Here brackets denote classes in $\operatorname{PGL}_2(\kk)$.

In the chosen bases, $[x:y]$ denotes $\kk(xc_i+yd_i)$ for all
$i\in\ZZ$ and all $(x,y)\in\kk^2\setminus\{(0,0)\}$.
Choose an unordered pair $\mathcal S_0$ of distinct lines. Since
$Z_i([x:y])=[x:-y]$ for all $i\in\ZZ$, its unique compatible
extension is $\mathcal S_i=S^i\mathcal S_0$ for all $i\in\ZZ$.
By \eqref{eq:repair-projective-recursion}, since $S^2 = \id$,
\[
 \begin{aligned}
 &g(\calS_i)=\calS_{i+r}\quad\text{for all }i\in\ZZ
\quad\Longleftrightarrow\quad (S^iM_0S^i)S^i\calS_0 = S^{i+r}\calS_0 \quad\text{for all }i\in\ZZ
\\
 \quad\Longleftrightarrow\quad &M_0\mathcal S_0=S^r\mathcal S_0
 \quad\Longleftrightarrow\quad T_g\mathcal S_0=\mathcal S_0.
 \end{aligned}
\]

\noindent
(\ref{item:monodromy-pairs}) $\Rightarrow$
(\ref{item:monodromy-diagonal}).
In a basis on the two lines of $\mathcal S_0$, the matrix of $T_g$ is
\[
 \begin{array}{cc}
 \text{both lines fixed}&\text{lines interchanged}\\[2pt]
 \begin{pmatrix}a&0\\0&b\end{pmatrix}&
 \begin{pmatrix}0&a\\b&0\end{pmatrix}
 \end{array}
 \qquad\text{for some }a,b\in\k^\times.
\]
By the field hypotheses, $X^2-ab$ has two distinct roots, so both
matrices are diagonalizable.

\noindent
(\ref{item:monodromy-diagonal}) $\Rightarrow$
(\ref{item:monodromy-pairs}).
Choose a basis $x_1,x_2$ of $\kk^2$ and
$\lambda_1,\lambda_2\in\k^\times$ such that
$T_gx_j=\lambda_jx_j$ for $j=1,2$.
Put
\[
 L_j=\kk\bigl((c_0\ d_0)x_j\bigr)\quad(j=1,2),
 \qquad \mathcal S_0=\{L_1,L_2\}.
\]
Then $L_1\ne L_2$ and $T_gL_j=L_j$ for $j=1,2$.
Thus $T_g\mathcal S_0=\mathcal S_0$, and the preceding
equivalence shows that the compatible system
$\mathcal S_i=S^i\mathcal S_0$ satisfies
\[
 g(\mathcal S_i)=\mathcal S_{i+r}
 \qquad\text{for all }i\in\ZZ.
\]

\noindent
(\ref{item:monodromy-pairs}) $\Leftrightarrow$
(\ref{item:monodromy-lift}).
Both implications and the standardness assertion follow from
Proposition~\ref{prop:repair-line-system}.

To prove the independence of the choices of bases and initial
index, recall that $r=g(0)$, where $0$ is the vertex
$0\in\ZZ=(\calC_\infty)_0$. Define
\[
 Z_{0,r}=\begin{cases}
 Z_{r-1}\cdots Z_0,&r>0,\\
 Z_r^{-1}\cdots Z_{-1}^{-1},&r<0,
 \end{cases}
 \qquad
 Z_{0,r}\colon\mathbb P(V_0)\longrightarrow\mathbb P(V_r).
\]
Write $\mathbb P(g)\colon\mathbb P(V_0)\to\mathbb P(V_r)$ for
$L\mapsto g(L)$. The map $Z_{0,r}^{-1}\mathbb P(g)$ is called the
{\emph{projective monodromy}} of $g$ at index $0$.
In the projective coordinates determined by the original
bases $(c_i,d_i)$, $Z_i$ is represented by $[S]$,
and $\mathbb P(g)$ by $[M_0]$.
Thus $Z_{0,r}^{-1}\mathbb P(g)$ is represented by
\[
 [S^{-r}M_0]=[S^rM_0]=[T_g],
\]
because $S^2=I$. Hence we obtain the commutative diagram
\begin{equation}\label{eq:repair-transport}
 \begin{tikzcd}[column sep=large]
 \mathbb P(V_0)\arrow[r,"\mathbb P(g)"]
 \arrow[rr,bend right=25,"{[T_g]}"']&
 \mathbb P(V_r)\arrow[r,"Z_{0,r}^{-1}"]&\mathbb P(V_0).
 \end{tikzcd}
\end{equation}
Keep the original repetitive bases $(c_i,d_i)$, and choose
another family of repetitive bases
$(\widetilde c_i,\widetilde d_i)$.
For every $i\in\ZZ$, define $N_i\in\operatorname{GL}_2(\kk)$
and $\eta_i\in\k^\times$ by
\[
 (\widetilde c_i\ \widetilde d_i)=(c_i\ d_i)N_i,
 \qquad
 \widetilde c_{i+1}\widetilde d_i
 =\widetilde d_{i+1}\widetilde c_i=\eta_iw_i.
\]
Because both families are repetitive bases, the composition
pairing gives
\[
 N_{i+1}^{\mathsf t}J_0N_i=\eta_iJ_0,
 \qquad
 [N_{i+1}]=[SN_iS],
 \qquad
 [N_i]=[S^iN_0S^i]
 \quad(i\in\ZZ).
\]
In the new bases, $g\colon V_i\to V_{i+r}$ has matrix $M'_i$,
and its projective monodromy at index $0$ is represented
by $T'_g$, where
\[
 M'_i=N_{i+r}^{-1}M_iN_i\quad(i\in\ZZ),
 \qquad T'_g=S^rM'_0.
\]
Consequently,
\[
 \begin{aligned}
 [T'_g]
 &=[S^rN_r^{-1}M_0N_0]\\
 &=[S^r(S^rN_0S^r)^{-1}M_0N_0]
  =[N_0^{-1}T_gN_0].
 \end{aligned}
\]
Next keep the original bases and take $j\in\ZZ$ as the
initial index. The projective monodromy at $j$ is represented
by $T_{g,j}:=S^rM_j$, and
\[
 [T_{g,j}]=[S^rM_j]
 =[S^rS^jM_0S^j]
 =[S^jT_gS^{-j}],
\]
because $S^{-j}=S^j$.
Scalar multiples and conjugation preserve diagonalizability,
so the last condition is independent of both choices.
\end{proof}

\begin{remark}\label{rem:repair-delta-powers}
For the automorphism $\delta$ in \eqref{eq:delta}, substitution in
\eqref{eq:basis} gives, for every $m\in\ZZ\setminus\{0\}$ and every
$i\in\ZZ$,
\[
 \delta^m(c_i)=(-1)^m c_{i+m},\qquad
 \delta^m(d_i)=d_{i+m}+m(-1)^{i+m}c_{i+m}.
\]
The projective monodromy of $\delta^m$ at index $0$, defined in the
proof of Proposition~\ref{prop:repair-monodromy}, is
\[
 [T_{\delta^m}]=
 \left[\begin{pmatrix}1&m\\0&1\end{pmatrix}\right]
 \quad\text{in }\operatorname{PGL}_2(\kk).
\]
Hence
\[
 \begin{aligned}
 &\delta^m\text{ preserves a compatible pair system}\\
 &\quad\Longleftrightarrow\quad T_{\delta^m}\text{ is diagonalizable}
 \quad\Longleftrightarrow\quad m1_{\kk}=0.
 \end{aligned}
\]
In characteristic zero, no nonzero power of $\delta$ preserves a
compatible pair system. In odd characteristic $p$, $\delta^m$
preserves one if and only if $p\mid m$; when $p\mid m$, it preserves
$\mathcal S_i=\{\kk c_i,\kk d_i\}$ for all $i\in\ZZ$.
Thus the failure for $\delta$ need not hold for all its nonzero powers.
These statements test compatibility with this family of grid coverings;
they do not assert necessity for standardness of arbitrary orbit algebras.
\end{remark}

\begin{remark}\label{rem:repair-scope}
Let $B$ be a Euclidean category in the sense of \cite[\S1.3]{Sko},
and let $G=\langle g\rangle$ be an admissible infinite cyclic group
of automorphisms of $\widehat B$. Assume that $g$ lifts through a
connected Galois covering $P:\calD\to\widehat B$ with simply connected
source. Theorem~\ref{thm:repair-lifts} then proves standardness of
$\widehat B/G$. This gives a sufficient replacement for the standardness
assertion in \cite[\S2.1]{Sko}, without a classification theorem.

For the repetitive Kronecker category, Proposition~\ref{prop:grid}
constructs the required source, and
Propositions~\ref{prop:repair-line-system} and
\ref{prop:repair-monodromy} give an explicit sufficient condition.
For a general Euclidean category in this sense, existence of such a
covering $P$ remains a hypothesis. These results give no necessary
criterion for standardness and do not address the other conclusions
of the Proposition in \cite[\S2.1]{Sko}.
\end{remark}

\section{The published equivalence and the standardness clause}
\label{sec:conclusion}

We have proved all parts of Theorem~A in the order stated.
The repetitive presentation and orbit-category calculation in
Section~\ref{sec:repetitive}, together with Lemma~\ref{lem:algebra},
prove (\ref{item:A-euclidean}), as recorded in Corollary~\ref{cor:euclidean}.
Proposition~\ref{prop:classification} and
Corollary~\ref{cor:nonstandard} prove (\ref{item:A-sources}).
Propositions~\ref{prop:grid} and \ref{prop:composite} prove (\ref{item:A-composite}).

The algebra $A$ is of Euclidean type and is not standard.
It therefore disproves the implication
\[
 R\text{ is of Euclidean type}
 \quad\Longrightarrow\quad
 R\text{ is standard}
\]
for selfinjective algebras $R$.
Consequently, it disproves the ``only if'' direction of the
equivalence between Euclidean type and standard
representation-infinite domestic type explicitly stated in
\cite[p.~45]{BHS04} and \cite[p.~158]{BS05}, both with reference
to \cite{Sko}. The reverse implication is not contradicted by
this example. Its failure of standardness already disproves the
stated implication; no conclusion about domesticity is needed
for this purpose.

The cited papers refer to \cite{Sko}. We can identify the logical
grounds of the equivalence by combining two assertions there.
The main theorem in \S1.5 gives the implication
from standard representation-infinite domestic type to Euclidean
type. The Proposition in \S2.1 asserts that the relevant
Euclidean orbit algebras are standard and representation-infinite
domestic. Combining these assertions gives the stated equivalence;
the error exposed here is in the standardness conclusion of the
Proposition.

The same equivalence, repeated in \cite[Remark~1.3]{AIS},
therefore needs correction. Its attribution there to the main
theorem alone is imprecise, but the substantive problem is the
false standardness conclusion just identified.
In particular, the equivalence cannot be used there to compare
Theorem~1.2 of \cite{AIS}, Theorem~1 of \cite{BHS04}, and
Theorem~2.1 of \cite{HS06}. Their hypotheses must instead be
compared separately. The present example by itself does not
disprove those classification statements as a whole.

The assertion directly contradicted in \cite{Sko} is more
specific. The Kronecker category satisfies the Euclidean
hypothesis in \cite[\S1.3]{Sko}: it is representation-infinite
and hereditary of Euclidean type, and its indecomposable
injectives form a complete slice in its preinjective component.
The hereditary property gives tiltedness, while the slice
verifies the additional condition in that definition.
The group $\langle\delta\rangle$ acts on $\widehat K$, is
infinite cyclic and free on objects, and has one object orbit.
It therefore satisfies exactly the admissibility condition in
\cite[\S1.4]{Sko}. Its orbit algebra is $A$.
The standardness conclusion of the Proposition in
\cite[\S2.1, p.~180]{Sko} would make $A$ standard, contrary to
Corollary~\ref{cor:nonstandard}.
The main theorem in \cite[\S1.5]{Sko} assumes standardness;
our nonstandard example is not a counterexample to that theorem
as stated. We make no claim here that it contradicts the other
conclusions of the Proposition.

Our orbit-algebra realization is for $\Omega(1)$ itself.
In \cite[proof of Proposition~2.2]{BHS}, the algebra realized as
an orbit algebra of a repetitive category is instead a different
algebra $A(n)$, which is socle equivalent to $\Omega(n)$.
{{\emph{Socle equivalence}} means that the quotients of the two algebras
by their socles are isomorphic.}
For $n=1$, the two generators $u,v$ of $A(1)$ satisfy
$u^2=v^2=0$ and $uv+vu=0$.
For all $R\in\{A(1),A\}$, put $J_R=\rad R$.
The displayed relations for $A(1)$ make all products of three
generators zero, so $J_{A(1)}^3=0$; also $J_A^3=0$ by
Lemma~\ref{lem:algebra}. Thus the {\emph{square map}}
\[
 q_R\colon J_R/J_R^2\longrightarrow J_R^2,\qquad
 z+J_R^2\longmapsto z^2
\]
is well-defined for both algebras; it need not be linear.
We have $q_{A(1)}=0$, whereas
$q_A(x\bar a+y\bar b)=y^2s$ for all $x,y\in\kk$.
Hence $A(1)\not\cong\Omega(1)=A$.
Our realization is not obtained by identifying these
socle-equivalent algebras.

The non-Galois composite $QP$ alone does not prove that $A$ is
nonstandard: another simply connected Galois covering could exist.
Proposition~\ref{prop:classification} excludes this possibility,
since none of its sources is simply connected. Thus the
nonstandardness proof is independent of the Euclidean realization
and repairs the gap examined in Section~\ref{sec:earlier}.

\section*{Acknowledgments}
This work was partially supported by JSPS Grant-in-Aid for
Scientific Research (C) 18K03207 and 25K06922, and by
Osaka Central Advanced Mathematical Institute (MEXT Promotion of
Distinctive Joint Research Center Program JPMXP0723833165).

\section*{Declaration of AI assistance}
The author used OpenAI's ChatGPT to assist with mathematical
discussions, the formulation and examination of proof arguments,
literature searches, and the preparation and revision of the text
and LaTeX source.
The author takes full responsibility for the mathematical claims,
proofs, references, and all other content of this paper.

\end{document}